\documentclass[12pt, a4paper, english]{amsart}
\title[Chow classes of matroids]{Chow classes of matroids and\\ standard Young tableaux}
\pdfpageattr{/Group <</S /Transparency /I true /CS /DeviceRGB>>}

\usepackage[T1]{fontenc}
\usepackage[english]{babel}
\usepackage{amsfonts,amsmath,amssymb,amsthm,mathrsfs,mathtools,amscd,amssymb}
\usepackage{dcolumn}
\usepackage{marvosym}
\usepackage{microtype}

\usepackage{mathtools}
\usepackage[dvipsnames]{xcolor}

\usepackage{xspace}
\usepackage{booktabs,tabularx}

\usepackage{array} % core package
\usepackage[backref=page,linktocpage]{hyperref} % permits navigating on the pdf
\usepackage{caption} %
\usepackage{subcaption}
\usepackage{graphics,graphicx} % permits putting images
\usepackage{tikz,tikz-cd} % tool for diagrams
\usetikzlibrary{graphs,graphs.standard,calc}
\usetikzlibrary{decorations.pathreplacing}%angles,quotes
\usetikzlibrary{decorations.pathmorphing,shapes,decorations.markings,decorations.text}
\usetikzlibrary{automata}
  \usetikzlibrary{through, backgrounds, intersections, perspective, 3d}
  \tikzstyle{rpath}=[rounded corners=0.1pt]
  \tikzstyle{dot}=[draw, shape=circle, fill=black, inner sep=0pt, minimum size=3pt]

\usepackage{enumerate} % permits enumerating using letters or other symbols
\usepackage{newtxtext,newtxmath} %fuente

\usepackage[margin=1.25in]{geometry}
\usepackage{verbatim}

\usepackage{scalerel}

\usepackage[backref=page,linktocpage]{hyperref} % permits navigating on the pdf
\hypersetup{
    colorlinks = true,
    linkbordercolor = {white},
    linkcolor = {BrickRed},
    anchorcolor = {black},
    citecolor = {BrickRed},
    filecolor = {cyan},
    menucolor = {red},
    runcolor = {cyan},
    urlcolor = {black}
}

\newtheoremstyle{theorems}% <name>
{13pt}% <Space above>
{13pt}% <Space below>
{\slshape}% <Body font>
{}% <Indent amount>
{\bfseries}% <Theorem head font>
{.}% <Punctuation after theorem head>
{.5em}% <Space after theorem headi>
{}% <Theorem head spec (can be left empty, meaning `normal')>

\theoremstyle{theorems}
\newtheorem{theorem}{Theorem}[section]
\newtheorem{corollary}[theorem]{Corollary}
\newtheorem{lemma}[theorem]{Lemma}
\newtheorem{conjecture}[theorem]{Conjecture}
\newtheorem{proposition}[theorem]{Proposition}
\newtheorem*{theorem*}{Theorem}
\newtheorem*{conjecture*}{Conjecture}

\newtheoremstyle{definition}% <name>
{12pt}% <Space above>
{12pt}% <Space below>
{}% <Body font>
{}% <Indent amount>
{\bfseries}% <Theorem head font>
{.}% <Punctuation after theorem head>
{.5em}% <Space after theorem headi>
{}% <Theorem head spec (can be left empty, meaning `normal')>

\theoremstyle{definition}
\newtheorem{definition}[theorem]{Definition}

\newtheorem{problem}[theorem]{Problem}

\newtheorem{example}[theorem]{Example}
\newtheorem{remark}[theorem]{Remark}

\usepackage{comment}

\newcommand{\NN}{\mathbb{N}}
\newcommand{\ZZ}{\mathbb{Z}}

\newcommand{\CC}{\mathbb{C}}

\DeclareMathOperator{\Sc}{Sc}
\DeclareMathOperator{\rmv}{rmv}
\DeclareMathOperator{\rank}{rank}
\DeclareMathOperator{\spn}{span}
\DeclareMathOperator{\supp}{supp}
\DeclareMathOperator{\des}{des}
\DeclareMathOperator{\Des}{Des}

\DeclareMathOperator{\SSYT}{SSYT}
\newcommand{\SYT}{\mathrm{SYT}}

\newcommand{\M}{\mathsf{M}}
\newcommand{\N}{\mathsf{N}}
\newcommand{\U}{\mathsf{U}}
\renewcommand{\S}{\mathsf{S}}
\DeclareMathOperator{\Volume}{V}
\DeclareMathOperator{\conv}{conv}
\DeclareMathOperator{\parallelExt}{P}
\DeclareMathOperator{\seriesExt}{S}

\DeclareMathOperator{\schurmatrix}{\Sigma}
\DeclareMathOperator{\KtoSc}{\mathcal L}
\DeclareMathOperator{\ribbon}{\rho}

\newcommand\polytope[1]{\mathscr{#1}}

\newcommand\SetOf[2]{\left\{\left.#1\vphantom{#2}\ \right|\ #2\vphantom{#1}\right\}}
\author[J.P.~Hamre]{Jon Pål Hamre}
\author[B.~Schr\"oter]{Benjamin Schr\"oter}
\author[L.~Vecchi]{Lorenzo Vecchi}
\author[E.~Verkama]{Emil Verkama}
\address{
  Department of Mathematics, KTH Royal Institute of Technology, Stockholm, Sweden
}
\email{\{jphamre, schrot, lvecchi, verkama\}@kth.se}

\thanks{
JPH is supported by the Swedish Research Council grant 2023-04063.
BS is supported by the Swedish Research Council grant 2022-04224.
EV is supported by the Swedish Research Council grant 2022-03875.
}

\usepackage{ytableau}
\usepackage[enableskew]{youngtab}

\colorlet{Red}{red!50}
\colorlet{Blue}{blue!50}
\colorlet{Yellow}{yellow!50}
\colorlet{Violet}{violet!50}


\begin{document}

\begin{abstract}
    We study the Chow classes of arbitrary matroids in the Grassmannian. We develop a new combinatorial approach for computing them, by first focusing on snake matroids and then extending our results via valuativity to any matroid. Our main contribution identifies the Poincaré dual of the Chow class of a snake matroid with a specific ribbon Schur function, providing an explicit formula for its coefficients in the Schubert basis as the number of standard Young tableaux of a given shape with a prescribed descent set. This agrees with a formula by Klyachko for the uniform matroid. As consequences, we recover and simplify classical results such as Gessel--Viennot’s enumeration of permutations with fixed descent sets, and formulas for the volume of lattice path matroids. Furthermore, we demonstrate the power of our findings by proving that certain Schubert coefficients are positive for all connected paving matroids. 
\end{abstract}
\maketitle

\section{Introduction} \noindent
The algebraic torus \(T=(\CC^*)^n\) acts naturally on the complex Grassmannian \(G(k,n)\) of \(k\)-subspaces of \(\CC^n\). To each point \(x\) in \(G(k,n)\) one may associate the \emph{torus orbit closure} \(\overline{Tx}\), i.e., the Zariski closure of the \(T\)-orbit of \(x\). It is a classical result of Gelfand, Goresky, MacPherson and Serganova \cite{GGMS} that \(\overline{Tx}\) is isomorphic to the projective toric variety of the matroid (base) polytope of the linear matroid \(\M_x\), which is represented by the columns of any matrix whose rowspace is \(x\).
In light of this result, it is natural to expect that properties of \(\overline{Tx}\) solely depend on the combinatorics of the matroid~\(\M_x\). One example of such a property is the class of \(\overline{Tx}\) in the Chow ring \(A^\bullet(G(k,n))\), as studied in \cite{FinkSpeyer}. This notion can be extended uniquely to a valuative invariant defined on arbitrary (not necessarily representable) matroids; see Remark~\ref{rem:valuation from rep} for further details. For an arbitrary matroid \(\M\) of rank \(k\) on a ground set of size \(n\), we call such a class \(\Sc(\M) \in A^\bullet(G(k,n))\) the \emph{Chow class} of the matroid \(\M\). Since the Schubert cycles form a basis of the ring \(A^\bullet(G(k,n))\), we may write
\[
    \Sc(\M) \ = \ \sum_{\eta} d_\eta(\M)\, \sigma_\eta \, ,
\]
where the sum ranges over partitions whose Young diagram is contained in the \(k \times (n-k)\) rectangle. The unique integers \(d_\eta(\M)\) appearing in this expansion are called the \emph{Schubert coefficients} of the matroid \(\M\). 

The Chow classes of matroids can be computed using equivariant \(K\)-theory \cite{FinkSpeyer,bergetfink_K} or tautological bundles \cite{BEST}. These methods are, however, relatively sophisticated and computationally expensive. 
This is one of the difficulties in the attempt to understand Chow classes, and it is therefore of interest to develop new techniques that are more direct and conceptually more transparent. A fundamental open question is the following.

\begin{conjecture*}[\cite{bergetfink_K}]
    For any matroid \(\M\) and partition \(\eta\), the Schubert coefficient \(d_\eta(\M)\) is nonnegative.
\end{conjecture*}

For representable matroids the Schubert coefficients count the number of intersection points of \(\overline{Tx}\) with certain Schubert varieties and are therefore nonnegative. In general, however, such a geometric interpretation is missing.
Nonnegativity has been verified only in a few nonrepresentable cases: when \(\M\) is \emph{sparse paving} \cite{JP}, when \(\eta^c\) is a \emph{hook} \cite{Speyer}, it has two rows and full first row, or it has two columns and full first column which follows from \cite[Section 9]{BEST}.

The central contribution of this paper is a purely combinatorial description of the Chow classes of matroids, which sheds new light on their properties and allows for drastically more efficient computations. First, we introduce a fundamental identity (Theorem \ref{thm:main identity}) that describes how Chow classes of matroids behave under series and parallel extensions. 
We use the identity to prove the following result (Theorem \ref{thm:ribbon snake}) for the Chow classes of \emph{snake matroids}. 

\begin{theorem*}
    Let \(\S\) be the snake matroid associated to the ribbon \(\ribbon\). Then the Poincaré dual of its Chow class is equal to the ribbon Schur function \(s_{\ribbon}\). 
\end{theorem*}

In the statement we exploited a standard isomorphism between the Chow ring \(A^\bullet(G(k,n))\) and a certain quotient of the ring of symmetric functions.
\emph{Ribbon Schur functions} are well studied classical objects, see for example \cite{macmahon_combinatory_analysis,king-welsh-vanwilligenburg,McNamara-Willigenburg}. A consequence of our result is that Schubert coefficients of snake matroids are Littlewood--Richardson coefficients (Corollary \ref{cor:Sc_coefficients_snakes_are_LR_ribbons}). This new insight, together with a theorem by Gessel \cite{gessel} can be used to prove that each Schubert coefficient \(d_\eta(\S)\) counts the number of standard Young tableaux with a prescribed \emph{descent set}. We also give a direct proof of this in Theorem \ref{thm:combinatorial Sc snakes}. For a partition \(\eta\) and a composition \(\mathbf b\) of length \(k\), let \(\SYT_\eta(\mathbf{b})\) denote the set of standard Young tableaux of shape \(\eta\) whose descent set is equal to \(\Des(\mathbf{b}) = \SetOf{\sum_{i = 1}^sb_i}{1 \leq s \leq k-1}\).

\begin{theorem*}
    Let \(\S\) be the snake matroid of rank \(k\) on \(n\) elements that is associated to the composition \(\mathbf{b}\) and let \(\eta\) be a partition contained in the diagram \(k\times (n-k)\). The Schubert coefficient of \(\S\) indexed by \(\eta\) is
    \[
        d_{\eta}(\S) \ = \ |\SYT_{\eta^c}(\mathbf{b})| \, ,
    \]
    where \(\eta^c\) denotes the complement of \(\eta\) in \(k\times (n-k)\).
\end{theorem*}

This result extends a formula for uniform matroids by Klyachko \cite{klyachko-r}, see also \cite{nadeau-tewari} for an independent proof. An alternative formula for the Chow class of a uniform matroid was studied by Berget and Fink \cite{berget-fink-equivariant}. A combinatorial interpretation of this formula that, in hindsight, could have hinted to our results can be found in \cite{lian}.
We are also able to give remarkably simple proofs of known results in the literature, such as the Gessel--Viennot enumeration of permutations with fixed descent set \cite{Gessel-Viennot} and formulas for the volume of lattice path matroids, in particular snake matroids \cite{Knauer-snakes, benedetti-knauer-valenciaporras}. Furthermore, with Proposition \ref{prop:support} we provide new bounds for the support of the Chow class of snake matroids or equivalently ribbon Schur functions.
Moreover, we extend our results to the larger class of \emph{lattice path matroids} (Theorem \ref{thm:combinatorial Sc lattice path}) and, together with results in \cite{Hampe}, provide a concrete computational algorithm for evaluating the Chow class of any matroid which we demonstrate in the Appendix~\ref{sec:appendix}.

We believe that this unexpected connection with the combinatorics of tableaux and symmetric functions can help with the nonnegativity conjecture of Schubert coefficients. We demonstrate this by proving positivity for specific shapes of \(\eta\) and the class of \emph{paving matroids} (Theorem \ref{thm: nonnegativity paving}). Our new expressions allow us to compute the Chow classes and verify the nonnegativity conjecture for all matroids up to \(8\) elements and all paving matroids up to \(15\) elements.

\subsection*{Outline} Section \ref{sec: preliminaries} introduces the necessary background on tableaux, matroids, valuations, and recalls the definitions of lattice path and snake matroids. We also review the definitions of the Chow class and of Schubert coefficients of a matroid.
Section \ref{sec:series parallel identity} establishes our key technical result, an identity relating the Chow class of a matroid with its series and parallel extensions (Theorem \ref{thm:main identity}).
In Section \ref{sec:chow class of snakes} we use Theorem \ref{thm:main identity} to compute the Chow classes of snake and lattice path matroids. We first show that the Poincaré dual of the Chow class of a snake matroid equals the corresponding ribbon Schur function, and then in Section \ref{sec:combinatorial} we derive a combinatorial formula expressing each of its Schubert coefficients as the number of standard Young tableaux with a prescribed descent set. Section \ref{sec:lpm} is dedicated to the extension of our results to lattice path matroids and thus also nested matroids and uniform matroids.
Section \ref{sec: applications} discusses consequences and applications of our results. We discuss the Gessel--Viennot enumeration of permutations with fixed descent sets, and volume formulas for lattice path matroids. We also derive new bounds for the support of a Chow class in Section \ref{sec:support} and partial positivity results for paving matroids in Section \ref{sec:positivity}. In Section \ref{sec:open problems} we present some open problems and in Appendix \ref{sec:appendix}, we show how to apply our results to compute Chow classes of arbitrary matroids.
We provide here a list of results on Schubert coefficients for various classes of matroids:
\begin{itemize}
    \item Snake matroids
    \begin{itemize}
        \item Ribbon Schur functions: Theorem \ref{thm:ribbon snake}
        \item Determinantal formula: Lemma \ref{lem:snake determinant}
        \item Littlewood--Richardson coefficients: Corollary \ref{cor:Sc_coefficients_snakes_are_LR_ribbons}
        \item Kostka numbers: Corollary \ref{cor:snake alternating formula}
        \item Standard Young tableaux: Theorem \ref{thm:combinatorial Sc snakes}
        \item Product of Schur functions: Proposition \ref{prop:Sc snake as a product}
    \end{itemize}
    \item Lattice path matroids: Theorem \ref{thm:combinatorial Sc lattice path}
    \item Nested matroids: Theorem \ref{cor:combinatorial Sc nested} 
    %\item Uniform matroids: Theorem \ref{thm:combinatorial Sc uniform} 
    \item Arbitrary matroids: Appendix \ref{sec:appendix}
\end{itemize}

\subsection*{Acknowledgments} The authors would like to thank Andrew Berget, Petter Brändén, Katharina Jochemko,  Krishna Menon, Felipe Rincón, Kris Shaw and for helpful comments and discussions during the writing process.
We also thank Carl Lian for pointing out Remark \ref{rmk:Carl Lian}, Philippe Nadeau and Vasu Tewari for Remark \ref{rmk:Vasu} and the reference to \cite{klyachko-r}, and Chris Eur for pointing us to the nonnegativity results in \cite{BEST}.

\section{Preliminaries and notation}\label{sec: preliminaries}
\subsection{Partitions, diagrams and tableaux}\label{sec:tableaux}
We begin by briefly summarizing the notions of partitions, skew shapes, diagrams and tableaux that we use in the remainder of this manuscript; most of it can be found in greater detail in \cite{Stanley-EC1,Stanley-EC2}.

We say that a sequence of positive integers \(\mathbf{b} = (b_1, b_2, \dots, b_k)\) with \(\sum_ib_i = n\) is a \emph{composition} of \emph{size} \(n\) and \emph{length} \(k\). If \(\lambda\) is a composition whose entries are ordered in nonincreasing order, we call it a \emph{partition}. We also write \(\lambda \vdash n\) and, as a shorthand notation, \(\lambda = [\ldots, 3^{m_3},2^{m_2},1^{m_1}]\) if the partition has \(m_i\) parts of size \(i\). For example,
\[
    \lambda = [5,3,2,2] \ = \  [5,3,2^2] \vdash 12 \, .
\]

The \emph{Young diagram} associated to a partition \(\lambda\vdash n\) is obtained by drawing a left-justified array of \(n\) boxes with \(\lambda_i\) boxes in row \(i\). Throughout we identify a partition with its Young diagram.
For positive integers \(p\) and \(q\) we denote the rectangular partition \([q^p]\) by \(p \times q\).
We say that a partition \(\mu\) is \emph{contained} in a partition \(\lambda\), and write \(\mu \subseteq \lambda\), if \(\mu_i \leq \lambda_i\) for every \(i\). When a partition \(\lambda\) is contained in \(p \times q\), we denote the \emph{complement} of \(\lambda\) in \(p\times q\) by \(\lambda^c = [\lambda^c_1, \lambda^c_2, \dots , \lambda^c_p]\) where \(\lambda^c_i = q - \lambda_{p+1-i}\). Given two partitions \(\mu \subseteq \lambda\), the \emph{skew Young diagram} \(\lambda/\mu\) is obtained by taking the set-theoretic difference of \(\lambda\) and \(\mu\), i.e., it is the set of boxes that belong to \(\lambda\) but not \(\mu\). A skew diagram where consecutive rows have at most one column in common is called a \emph{ribbon}. Ribbons will play a fundamental role in the rest of the paper. To any composition \(\mathbf b\) we can associate a ribbon \(\ribbon(\mathbf b)\) whose \(i\)-th row from the bottom has length \(b_i\), for all \(i\). Figure \ref{fig:young diagram} depicts two Young diagrams and the corresponding skew Young diagram, which turns out to be a ribbon. 

\begin{figure}[ht]
    \centering
    \scalebox{0.8}{\begin{tikzpicture}
    \node[anchor=south west] at (0,0) {
      \ydiagram{5,3,2,2}
    };
    \node[anchor=south west] at (4.5,0.65) {
      \ydiagram{2,1,1}
    };
    \node[anchor=south west] at (7,0) {
      \ydiagram{2+3,1+2,1+1,2}
    };
\end{tikzpicture}}
    \caption{The Young diagrams of the partitions \(\lambda = [5,3,2^2] \vdash 12\) on the left, \(\mu = [2,1^2] \vdash 4\) in the center and the ribbon \(\ribbon(2,1,2,3)\) on the right.}
    \label{fig:young diagram}
\end{figure}
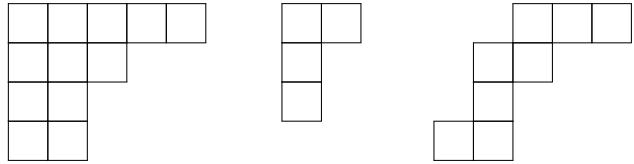

A \emph{semistandard Young tableau} of shape \(\lambda\) is a filling of the Young diagram of \(\lambda\) with positive integers that is weakly increasing over rows and strictly increasing over columns. We denote by \(\SSYT_\lambda\) the set of semistandard Young tableaux of shape \(\lambda\), and we use the notation \(\SSYT_\lambda({\leq} k)\) for the set of all semistandard Young tableaux whose filling is selected from the numbers \([k] = \{1,2,\ldots,k\}\). A \emph{standard Young tableau} is a semistandard Young tableau which uses all numbers from \([n]\) where \(n = |\lambda|\). We denote the set of standard Young tableaux of shape \(\lambda\) by \(\SYT_\lambda\). Similarly, one may define (semi)standard skew tableaux for diagrams of shape \(\lambda/\mu\).

Given a partition \(\lambda\), the \emph{symmetric Schur function} \(s_\lambda\) is the function over infinitely many variables \(u_1,u_2,\ldots\) , defined by
\[
    s_\lambda \ = \sum_{T \in \SSYT_\lambda}u^T \, ,
\]
where \(u^T = u_1^{t_1}u_2^{t_2}\ldots u_k^{t_k}\) and the exponent \(t_i\) is the number of boxes in \(T\) that have an entry equal to \(i\). 
When restricting the sum to semistandard Young tableaux in \(\SSYT_\lambda({\leq}k)\), we get the \emph{symmetric Schur polynomial in \(k\) variables},
which, by abuse of notation, we denote also by \(s_\lambda = s_\lambda(u_1,\ldots, u_k)\). It is known that the Schur functions (resp. polynomials) form a basis for the ring of symmetric functions (resp. polynomials). In particular, we may write 
\begin{equation}
    \label{eq:mult of schur functions}
    s_\mu s_\eta \ = \ \sum_{\lambda}c^\lambda_{\mu,\eta}s_\lambda \ .
\end{equation}
The coefficients \(c^\lambda_{\mu,\eta}\) are called \emph{Littlewood--Richardson coefficients} and are in general hard to compute. 
When \(\eta = [b]\), respectively \(\eta = [1^b]\), the formula simplifies greatly to what is known as Pieri's formula.
\begin{theorem}[Pieri's formula]
    \label{thm:pieri}
    Let \(\lambda\) be a partition and \(b \in \NN\). Then
    \[
    s_\mu s_{[b]} \ = \ \sum_{\lambda}s_\lambda \ ,
    \]
    where the sum runs over all partitions whose diagram is obtained from \(\mu\) by adding \(b\) boxes, at most one per column. Similarly,
    \[
        s_\mu s_{[1^b]} = \sum_{\lambda} s_\lambda\ ,
    \]
    where the sum runs over all partitions whose diagram is obtained from \(\mu\) by adding \(b\) boxes, at most one per row.
\end{theorem}
Similarly, one can define the \emph{skew Schur function} \(s_{\lambda/\mu}\). It is a well-known result that when expanding \(s_{\lambda/\mu}\) in the Schur basis, one obtains
\[
s_{\lambda/\mu} = \sum_{\eta}c^\lambda_{\mu,\eta}s_\eta \,.
\]
\subsection{Matroids}\label{sec:matroids}
We assume familiarity with basic concepts from matroid theory.
Here, we recall only some chosen facts in order to set up the notation used throughout the paper. For any undefined concept we refer the reader to \cite{Oxley} and to \cite{Hampe} for the background on cyclic flats. There are many ways to represent a matroid. For example, a matroid \(\M\) is given by its \emph{ground set}, that we denote by \(E\), and its \emph{bases} by \(\mathcal{B}\subseteq 2^E\) which satisfy an exchange axiom. 
It is known that every basis of a matroid has the same cardinality, which is called the \emph{rank} of the matroid \(\M\). 
\begin{remark}
    We often need an order on the ground set and it is convenient to us to use \(E = [n] = \{1,2,\ldots,n\}\) as the ground set with its natural ordering.
    Moreover, for this reason we do not identify isomorphic matroids.
\end{remark}

We denote the set of all matroids by \(\mathcal{M}\) and all matroids of given rank \(k\) and ground set \([n]\) by \(\mathcal{M}_{n,k}\). 
One of the most basic examples for a matroid is the \emph{uniform matroid}~\(\U_{k,n}\) of rank \(k\) on \(E=[n]\) whose bases include all subsets of \([n]\) of size \(k\).
A standard operation on matroids is the \emph{direct sum}.
\begin{definition}
    Let \(\M_1,\ldots, \M_m\) be a family of matroids on disjoint ground sets \(E_i\) with bases \(\mathcal B_i\).
    The \emph{direct sum} \(M_1 \oplus \dots \oplus M_m\) of these matroids is a matroid on \(E=\bigsqcup_i E_i\) and bases \(\mathcal B = \{ \bigsqcup_i B_i \mid B_i \in \mathcal B_i \}\).
\end{definition}

\begin{remark}
    As the order on the ground set plays a key role we follow the following convention, when not stated otherwise.
    If \(\M_1\) is a matroid on \(E_1=[n_1]\) and \(\M_2\) is a matroid on \(E_2=[n_2]\) then \(\M = \M_1\oplus\M_2\) is a matroid on \(E=[n_1+n_2]\) with \(\M|_{[n_1]}=\M_1\) and \(\M|_{\{n_1+1,\ldots,n_1+n_2\}}=\M_2\). Notice that \(\M_1\oplus\M_2\) and \(\M_2\oplus\M_1\) are not the same matroids, but they are isomorphic.
\end{remark}

A matroid \(\M\) is \emph{disconnected} if it can be written as a direct sum of matroids with nonempty ground sets; otherwise, \(\M\) is \emph{connected}. Every matroid can be written in a unique way as a direct sum of connected matroids, whose ground sets are called the \emph{connected components} of the matroid \(\M\). We denote by \(\kappa(\M)\) the number of connected components of \(\M\).

If \(\M = \N \oplus \U_{1,1}\) then the last element is a coloop and in \(\M = \N \oplus\U_{0,1}\) it is a loop; both of these matroids are disconnected whenever \(\N\) has a nontrivial ground set. 

We need two further operations on matroids, which we now recall.
\begin{definition}
    Let \(\M\) be a matroid of rank \(k\) on \([n]\). 
    Adding the element \(n+1\) in parallel to an element \(e\) is the \emph{parallel extension} \(\parallelExt_e(\M)\) of \(\M\) on the ground set \(E=[n+1]\) and bases
    \[
        \mathcal B(\M) \cup \SetOf{B \setminus \{e\} \cup \{n+1\}}{B \in \mathcal{B}(\M), \, e \in B}\, .
    \]
    The dual operation is a coextension \(\seriesExt_e(\M)\), called the \emph{series extension} of \(\M\). Its ground set is \(E=[n+1]\) and its set of bases is
    \[
        \SetOf{B \cup \{n+1\}}{B \in \mathcal B(\M)} \cup \SetOf{B \cup \{e\}}{B \in \mathcal B(\M), \, e \notin B}\, .
    \]
\end{definition}
\begin{remark}
     We usually consider the parallel and series extension with respect to the element \(e=n\). In these cases we write \(\parallelExt(\M)\) instead of \(\parallelExt_e(\M)\) and \(\seriesExt(\M)\) instead of \(\seriesExt_e(\M)\).
\end{remark}

\subsection{Lattice path, nested and snake matroids} The most important examples of matroids in this article fall into the class of \emph{lattice path matroids}, which were first introduced in \cite{BoninDeMier}.
Let \(k\leq n\) and consider lattice paths consisting of north steps \(\mathbf N = (0,1)\) and east steps \(\mathbf E = (1,0)\) from \((0,0)\) to \((k,n-k)\).
Given such a lattice path \(T\), denote by \(B(T)\) the set of indices of \(T\) which are north steps.
Given two lattice paths \(\mathbf L\) and \(\mathbf U\) such that \(\mathbf U\) is always above \(\mathbf L\), consider the collection
    \[
        \mathcal B(\mathbf L,\mathbf U)\ = \ \SetOf{B(T)}{T \text{ is always between \(\mathbf L\) and \(\mathbf U\)}}\, .
    \]
The collection \(\mathcal B(\mathbf L, \mathbf U)\) forms the bases of a matroid of rank \(k\) on \([n]\) that we denote by \(\M[\mathbf L, \mathbf U]\). 
A matroid that arises in this way is called a \emph{lattice path matroid}. All lattice path matroids are representable over infinite fields. A matroid that is isomorphic to a lattice path matroid with \(\mathbf U = \mathbf N^k\mathbf E^{n-k}\) is called a \emph{nested matroid}. In particular a lattice path matroid with \(\mathbf L = \mathbf E^{n-k}\mathbf N^k\) is a nested matroid. We denote the family of all nested matroids of rank \(k\) on \([n]\) by \(\mathcal{N}_{n,k}\). 

\begin{example}
    The uniform matroid \(\U_{k,n}\) is the lattice path matroid with \(\mathbf L = \mathbf E^{n-k}\mathbf N^k\) and \(\mathbf U = \mathbf N^k\mathbf E^{n-k}\). This matroid is nested.
\end{example}
A connected lattice path matroid of rank \(k\) on \([n]\) can also be encoded by a skew shape \(\lambda /\mu\), where \(\lambda\) has \(k\) parts and \(\lambda_1 = n-k\), thus it is convenient for us to use the alternative notation \(\M(\lambda/\mu)\). A \emph{snake matroid} is a connected lattice path matroid whose associated skew shape is a ribbon. 
For a composition \(\mathbf b\) we write \(\S(\mathbf b)\) for the corresponding snake matroid \(\M(\rho(\mathbf b))\).
  
\begin{example}\label{ex:snake} Consider the snake matroid \(\S(2,1,2,3)\) of rank four on nine elements depicted in Figure~\ref{fig:snake}. It is the lattice path matroid \(\M(\lambda/\mu)\) of the ribbon \(\lambda/\mu\) where \(\lambda = [5,3,2^2]\) and  \(\mu = [2,1^2]\); see Figure~\ref{fig:young diagram}. 
\end{example}

\begin{figure}[ht]
    \centering
    \begin{tikzpicture}
\node[anchor=south west] at (5,0) {
    \begin{ytableau}
    \none & \none & *(White)  & *(White) & *(White)\\
    \none & *(White) & *(White) & \none & \none\\
    \none & *(White) & \none & \none    & \none \\
    *(White)& *(White)& \none & \none    & \none \\
    \end{ytableau}
};
    \draw[Red, line cap=round, line join=round, line width=1.3 pt] (5.15,0.15) -- (5.15,0.8) -- (5.8,0.8) -- (5.8,2.1) -- (6.45,2.1) -- (6.45,2.75) -- (8.40,2.75);
    \draw[Blue, line cap=round, line join=round, line width=1.3 pt] (5.15,0.15) -- (6.45,0.15) -- (6.45,1.45) -- (7.10,1.45) -- (7.10,2.10) -- (8.40,2.10) -- (8.40,2.75);
\end{tikzpicture}
    \caption{The snake matroid \(\S(2,1,2,3)\) from Example~\ref{ex:snake}. We highlighted the upper path \(\mathbf{U}= \mathbf{NENNENEEE}\)\ (red) and lower path \(\mathbf{L}= \mathbf{EENNENEEN}\) (blue).}
    \label{fig:snake}
\end{figure}
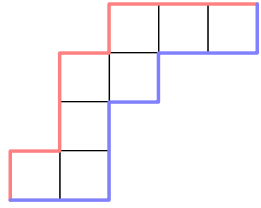

\begin{remark}\label{rem:snake series parallel}
    Snake matroids are exactly the matroids that can be built via a successive sequence of series and parallel extensions with respect to the last element starting from \(\U_{1,2}\). Indeed, we have
    \[
        \S(b_1,\ldots,b_{k-1},b_k) \ = \ \parallelExt^{b_{k}-1}(\seriesExt(\S(b_1,\ldots,b_{k-1}))
    \]
    for all \(k>1\), and \(\S(b_1) = \parallelExt^{b_1-1}(\U_{1,2})\).
\end{remark}

\subsection{Valuations}
Next, we recap the notion of being a valuative function for matroids. To introduce this concept we associate the following \(0/1\)-polytope with the matroid \(\M\)
\[
    \polytope{P}(\M) \ = \ \conv\SetOf{e_B\in\mathbb{R}^E}{B\in\mathcal{B}(\M)}\, ,
\]
where \(e_B = \sum_{i\in B} e_i\) is a sum of standard basis vectors.
This polytope is called the \emph{matroid (base) polytope}. 
A \emph{matroid subdivision} \(\Sigma\) is a polyhedral complex that covers a polytope \(\polytope{P}\) where every cell is a matroid polytope.
A \emph{valuation} \(f:\mathcal{M}\to G\) is a function on the space of all matroids to an abelian group~\(G\) which satisfies the inclusion-exclusion principle, that is for any matroid subdivision \(\Sigma\) of the matroid polytope \(\polytope{P}\) 
\[
    f(\polytope{P})\ = \ \sum_{\polytope{Q}\in\Sigma^{\text{int}}} (-1)^{\dim \polytope{P} - \dim \polytope{Q}} f(\polytope{Q})\, ,
\]
where \(\Sigma^{\text{int}} = \SetOf{\polytope{Q}\in\Sigma}{\polytope{Q}\not\subseteq \partial \polytope{P}} \) and \(\partial \polytope{P}\) is the boundary of the polytope \(\polytope{P}\).

\begin{remark}
    There is a variety of related concepts for polyhedra see \cite[Appendix~A]{EurHuhLarson}, which are either equivalent to our notion by \cite{DerksenFink} or implied by it.  
\end{remark}
We say that a matroid \(\M\) decomposes into \(\N_1,\ldots,\N_j\) if \(\{\polytope{P}(\N_1),\ldots, \polytope{P}(\N_j)\}\) is a matroid subdivision of \(\polytope{P}(\M)\) and we write \(f(\M)\) as a shorthand notation for \(f(\polytope{P}(\M))\).
We call a valuation \(f\) \emph{additive} if \(f(\M)=0\) whenever the matroid~\(\M\) is disconnected.

\begin{example}[\cite{Crapo:67}]\label{ex:beta_invariant}
    Crapo's well known \emph{\(\beta\)-invariant}
    \[
    \beta(\M) \ = \ (-1)^k \sum_{S\subseteq E} (-1)^{|S|} \rank(S)
    \]
    is an example of an additive valuative matroid invariant as it is the linear term of the Tutte polynomial, yet another important valuative matroid invariant; see \cite{speyer_tropical}. 
\end{example}

The following theorem highlights the role of nested matroids for valuations.
\begin{theorem}[{\cite[Theorem 5.4]{DerksenFink}}]\label{thm:bases_valuation}
    For every matroid \(\M\) of rank \(k\) on \([n]\) there exist unique integers \((c(\N,\M))_{\N\in\mathcal{N}_{n,k}}\) such that
    \[
    \sum_{\N\in\mathcal{N}_{n,k}} c(\N,\M) f(\N) \ = \ f(\M)
    \]
    for all valuations \(f\).
\end{theorem}
Moreover, Hampe \cite{Hampe} described an algorithm to compute the coefficients \(c(\N)\) explicitly in terms of the M\"obius function on the poset of chains of cyclic flats of \(\M\).

\begin{remark}\label{rem:valuation from rep}
     Suppose \(f\) is a map defined on all representable matroids, which satisfies the relations \(\sum_{\N\in\mathcal{N}_{n,k}} c(\N,\M) f(\N) \ = \ f(\M)\). Then it is natural to extend this map to a valuation on all matroids. This extension is unique. It follows that if two valuations are equal on all representable matroids, they are equal on all matroids. 
\end{remark}

We can further decompose every lattice path matroid by recursively considering the lowest and highest path passing trough a inner point; see \cite[Appendix A]{FerroniSchroeter:2024} and \cite[Lemma 4.3.5]{Bidkhori:2010}. When \(f\) is an additive function, this decomposition leads to the following result. 
\begin{proposition}
    \label{prop:valuative into snakes}
    Let \(f\) be an additive valuation and \(\M(\lambda/\mu)\) a lattice path matroid, then
    \[
    f(\M(\lambda/\mu)) \ = \ \sum_{\S} f(\S) \, ,
    \]
    where the sum is taken over all snake matroids \(\S\) that fit into the skew shape \(\lambda/\mu\). 
\end{proposition}

\begin{remark}
    For any additive, valuative matroid invariant \(f\), 
    as every nested matroid is isomorphic to a lattice path matroid, we can express \(f(\M)\) as a linear combination of evaluations on snake matroids.
\end{remark}

\begin{example}
Let \(\M[\mathbf{L},\mathbf{U}]\) be the rank three lattice path matroid with lower path \(\mathbf{L} = \mathbf{E}^4\mathbf{N}^3\) and upper path \(\mathbf{U} = \mathbf{N}\mathbf{E}\mathbf{N}\mathbf{E}\mathbf{E}\mathbf{N}\mathbf{E}\).
Figure~\ref{fig:LatticePathDecomp} depicts \(\M[\mathbf{L},\mathbf{U}]\) and all the snake matroids fitting inside \(\lambda/\mu\). This implies that
\[
f(\M[\mathbf{L},\mathbf{U}]) = f(\S(2,3,1)) + f(\S(3,2,1)) +  f(\S(4,1,1))
\]
for any additive valuation \(f\).
\end{example}
\begin{figure}[ht]
    \centering
    \begin{tikzpicture}
\node[anchor=south west] at (-6,-3) {
    \begin{ytableau}
    \none & \none & \none  & *(White)\\
    \none & *(White) & *(White) & *(White)\\
    *(White) & *(White) & *(White) & *(White)\\
    \end{ytableau}
};
\node[black] at (-4.45,-3.2) {\(\M[\mathbf{L},\mathbf{U}]\)};

\node[anchor=south west] at (-1.5,-3) {
    \begin{ytableau}
    \none & \none & \none  & *(White)\\
    \none & *(White) & *(White) & *(White)\\
    *(White) & *(White) & \none & \none\\
    \end{ytableau}
};
\node[black] at (-.05,-3.2) {\(\S(2,3,1)\)};

\node[anchor=south west] at (1.5,-3) {
    \begin{ytableau}
    \none & \none & \none    & *(White)\\
    \none & \none & *(White) & *(White)\\
    *(White) & *(White) & *(White) & \none\\
    \end{ytableau}
};
\node[black] at (2.95,-3.2) {\(\S(3,2,1)\)};

%\node[anchor=south west] at (7.5,-3) {
%    \begin{ytableau}
%    \none & \none & *(White)  & *(White)\\
%    \none & \none & *(White) & \none\\
%    *(White) & *(White) & *(White) & \none\\
%    \end{ytableau}
%};
%\node[black] at (8.95,-3.2) {\(\S(3,1,2)\)};

\node[anchor=south west] at (4.5,-3) {
    \begin{ytableau}
    \none & \none & \none & *(White)\\
    \none & \none & \none & *(White)\\
    *(White) & *(White) & *(White) & *(White)\\
    \end{ytableau}
};
\node[black] at (5.95,-3.2) {\(\S(4,1,1)\)};

\end{tikzpicture}
    \caption{A lattice path matroid and the corresponding snake matroids.}
    \label{fig:LatticePathDecomp}
\end{figure}
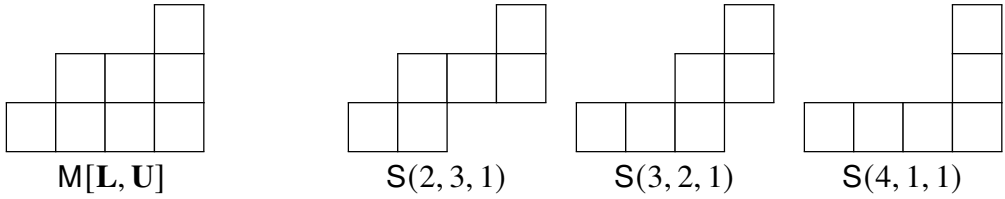

\subsection{The Chow class of a matroid}\label{subsec:schubert coefficients}
For nonnegative integers \(k \leq n\) we denote the Grassmannian of \(k\)-dimensional linear subspaces of \(\CC^n\) by \(G(k,n)\). 
For any point \(x \in G(k,n)\) we denote by \(\M_x\) the matroid of rank \(k\) on \([n]\) represented by the columns of a matrix whose rowspace is \(x\). That is, the bases of \(\M_x\) are the sets \(B \in \binom{[n]}{k}\) such that the Plücker coordinate \(p_B(x)\) is nonzero. The algebraic torus \(T = (\CC^*)^n\) acts on \(G(k,n)\) by scaling the standard coordinates of \(\CC^n\). We are interested in the \emph{torus orbit closure} of \(x\), i.e., the Zariski closure of the \(T\)-orbit of \(x\), denoted by \(\overline{Tx}\). The study of the relationship between \(\overline{Tx}\) and the matroid~\(\M_x\) was initiated by Gelfand, Goresky, MacPherson and Serganova in \cite{GGMS}. The authors showed that \(\overline{Tx}\) is isomorphic to the projective toric variety given by the matroid base polytope \(\polytope{P}(\M_x)\). In light of this result it is natural to expect that properties of \(\overline{Tx}\) depend solely on the matroid~\(\M_x\). An example is the class \([\overline{Tx}]\) in the Chow ring \(A^\bullet(G(k,n))\), which is the main object of interest in the remainder of this article.
The Chow ring \(A^\bullet(G(k,n))\) is generated as a free abelian group by the \emph{Schubert cycles}  \(\sigma_\eta\), indexed by partitions \( \eta \subseteq k\times(n-k)\),
\[
    A^\bullet(G(k,n)) \ = \ \ZZ\SetOf{\sigma_\eta }{\eta \subseteq k \times (n-k)} \, .
\]
Schubert cycles are classes of Schubert varieties \(\sigma_\eta = [X_\eta(V_\bullet)]\), where \(X_\eta(V_\bullet)\) is the Schubert variety associated to the partition \(\eta\) and a complete flag \(V_\bullet\) of \(\CC^n\), see \cite{EisHar} for further details.
This Chow ring is graded by codimension and the top degree part \(A^{k(n-k)}(G(k,n))\) is isomorphic to the integers~\(\ZZ\). The \emph{degree map} \(\deg \colon A^{k(n-k)}(G(k,n)) \to \ZZ\) is the isomorphism normalized via \(\deg(\sigma_{k\times (n-k)}) = 1\). If \(Y\) is a zero-dimensional subvariety of \(G(k,n)\), that is a finite set of points, then \(\deg([Y])\) counts the number of points in \(Y\). 

We prefer to work with symmetric functions instead of Schubert cycles. To do so, we consider the ring \(R = R(k,n)\), given as the quotient of the ring of symmetric functions by the ideal \(\langle s_\eta \mid \eta \not \subseteq k\times (n-k)\rangle\), which is isomorphic to \(A^\bullet(G(k,n))\) via the map \(\sigma_\eta \mapsto s_\eta\) where \(s_\eta\) is the Schur function of the partition \(\eta\). In Section~\ref{sec:series parallel identity} below we specialize further to the corresponding ring of Schur polynomials in \(k\) variables, but we will need the more general concept of symmetric functions in Section~\ref{sec:chow class of snakes}.
This isomorphism allows us to express the multiplicative structure of the Chow ring \(A^\bullet(G(k,n))\) in terms of the combinatorics of partitions. More precisely, one may write a product of Schubert cycles as the (unique) linear combination
\[
    \sigma_\mu \sigma_\eta \ = \ \sum_{\lambda} c_{\mu, \eta}^\lambda \ \sigma_\lambda \, ,
\]
where the sum is taken over all partitions \(\lambda \subseteq k\times (n-k)\) with \(|\lambda| = |\mu| + |\eta|\) and the coefficients are the Littlewood--Richardson coefficients as in \eqref{eq:mult of schur functions}. In addition to the analogue of Pieri's formula, see Theorem \ref{thm:pieri}, we also have the \emph{complimentary dimension formula}; if \(|\mu| + |\eta| = k(n-k)\) then \(\deg(\sigma_\mu \sigma_\eta) = 1\) if \(\eta = \mu^c\) and \(0\) otherwise.

If \(Y\) is a subvariety of \(G(k,n)\) of codimension~\(m\), then we are able to write the Chow class \([Y] \in A^m(G(k,n))\) as a linear combination \([Y] = \sum_\eta d_\eta \sigma_\eta\), where the sum is taken over all partitions \(\eta \vdash m\) with \(\eta \subseteq k\times (n-k)\). Here the coefficient \(d_\eta = \deg([Y]\sigma_{\eta^c})\) is the number of points in the intersection of varieties \(Y \cap X_{\eta^c}(V_\bullet)\), where \(V_\bullet\) is a generic complete flag of~\(\CC^n\). 
The torus orbit closure \(\overline{Tx}\) is of dimension \(n - \kappa(\M_x)\), hence of codimension \( m \coloneqq k(n-k) - (n-\kappa(\M_x))\). Thus we may write
\[
    [\overline{Tx}] \ = \sum_{\substack{\eta \vdash m \\ \eta \subseteq k \times (n-k)}} d_\eta \sigma_\eta \ \in \ A^m(G(k,n)) \, .
\]
It follows from \cite[Proposition A.5]{Speyer} that the class \([\overline{Tx}]\), and each coefficient \(d_\eta\) only depends on the isomorphism class of the matroid \(\M_x\). The coefficients \(d_\eta\) are called the \emph{Schubert coefficients} of the representable matroid \(\M_x\). 
The map \(\M_x\mapsto [\overline{Tx}]\) is a valuation and can, as in Remark~\ref{rem:valuation from rep}, be extended uniquely to a valuation defined on all matroids. In \cite{FinkSpeyer}, this is made precise by defining the class of a matroid \(\M\) in the \(K\)-theory of the Grassmannian; whenever \(\M = \M_x\), this \(K\)-class agrees with the class of the structure sheaf of \(\overline{Tx}\) and specializes to a class in \(A^\bullet(G(k,n))\) that agrees with \([\overline{Tx}]\). We summarize the results from \cite{FinkSpeyer} that are most relevant for us in the next theorem.
\begin{theorem}[\cite{FinkSpeyer}]
    \label{thm:FinkSpeyer}
    Fix integers \(0 \leq k \leq n\) and \((k-1)(n-k-1) \leq m \leq  k(n-k)\). There exists a unique valuative matroid invariant
    \[
        \Sc_m \colon \mathcal M_{n,k} \longrightarrow A^m(G(k,n)) \, ,
    \]
    such that \(\Sc_m(\M) = [\overline{Tx}]\) whenever \(\M = \M_x\) and \(m = k(n-k)-(n-\kappa(\M))\), and is zero if \(m \neq k(n-k)-(n-\kappa(\M))\).
\end{theorem}

\begin{definition}
    \label{def:Schubert coefficients}
    Let \(\M\) be a matroid of rank \(k\) on \([n]\), and \(m = k(n-k) - (n-\kappa(\M))\). The \emph{Schubert coefficients} of the matroid \(\M\) are the integers \(d_{\eta, m}(\M)\) such that
    \[
        \Sc_m(\M) \ = \ \sum_\eta d_{\eta,m}(\M) \sigma_\eta \, .
    \]
\end{definition}
We also consider the Poincaré dual of the class \(\Sc_m(\M)\) and use the notation
\[
    \Sc_m^c(\M) \ = \ \sum_{\eta} d_{\eta, m}(\M) \sigma_{\eta^c} \, .
\]
It follows from Theorem \ref{thm:FinkSpeyer} that for a fixed partition \(\eta\) and integers \(0 \leq k \leq n\) and \((k-1)(n-k-1) \leq m \leq k(n-k)\), the Schubert coefficient
\[
    d_{\eta, m} \colon \mathcal M_{n,k} \longrightarrow \ZZ
\]
is a valuative matroid invariant. Moreover, \(d_{\eta,m}(\M) = 0\) if \(\eta \not\subseteq k\times (n-k)\), \(|\eta| \neq m\) or \(m \neq k(n-k) - (n - \kappa(\M))\). We suppress the subscript \(m\), when it does not cause any confusion. That is, for a matroid \(\M\) of rank \(k\) on \([n]\) we write \(\Sc(\M) = \Sc_{k(n-k) - (n-\kappa(\M))}(\M)\) and \(d_\eta(\M) = d_{\eta, k(n-k) - (n-\kappa(\M))}(\M)\). In particular whenever the matroid \(\M\) is connected we write \(\Sc(\M) = \Sc_{(k-1)(n-k-1)}(\M)\) and \(d_\eta(\M) = d_{\eta,(k-1)(n-k-1)}(\M)\). In this case if \(|\eta| = (k-1)(n-k-1)\), then \(d_\eta(\M)\) is zero for disconnected matroids. That is, \(d_{\eta, (k-1)(n-k-1)}\) is an additive and valuative matroid invariant. When a matroid \(\M\) of rank \(k\) on \([n]\) is representable, there is some point \(x \in G(k,n)\) such that \(\M = \M_x\). In this case we see that the Schubert coefficients  of \(\M\) are nonnegative integers, as they count the number of points in an intersection of varieties. A main open problem about Schubert coefficients is whether this holds for arbitrary matroids.
\begin{conjecture}[{\cite[Conjecture 9.13]{bergetfink_K}}]
    \label{conj:nonnegativity}
    For any matroid \(\M\) and partition \(\eta\), the Schubert coefficient \(d_{\eta}(\M)\) is nonnegative.
\end{conjecture}

Not much is known about Schubert coefficients, and computations are generally difficult. In \cite{klyachko-r}, see Corollary \ref{cor:combinatorial Sc uniform}, Klyachko gives a formula for the Schubert coefficients of uniform matroids.

In \cite[Theorem 5.1]{Speyer} Speyer recognizes one of the Schubert coefficients as the \(\beta\)-invariant of Example~\ref{ex:beta_invariant}. Specifically, he shows that for a matroid \(\M\) of rank~\(k\)~on~\([n]\)
\[
    d_{h^c, (k-1)(n-k-1)}(\M) \ = \ \beta(\M) \, ,
\]
where \(h\) is the hook partition \(h = [n-k,1^{k-1}]\).

It is known that the \(\beta\)-invariant is nonnegative for all matroids \cite{Crapo:67}, hence this confirms Conjecture \ref{conj:nonnegativity} for \(\eta = h^c\). The proof of Speyer relies on the fact that \(d_{h^c}\) satisfies the deletion-contraction principle, but this is not true in general for other partitions. 

\begin{example}
    \label{ex:U25}
    Using the result by Klyachko, see Corollary \ref{cor:combinatorial Sc uniform}, we can compute the Chow class of the uniform matroid \(\U_{2,5}\), 
    \[
        \Sc(\U_{2,5}) \ = \ 3\sigma_{\scalebox{0.3}{\young(~~)}} + 1\sigma_{\scalebox{0.3}{\young(~,~)}} \, .
    \]
    Here the complement of the hook partition is \(h^c = [2]\), which agrees with Speyer's result as \(\beta(\U_{2,5}) = 3\).
\end{example}
In \cite[Theorem 1.1]{JP} the first author showed that all the Schubert coefficients of connected sparse paving matroids, except for \(d_{h^c}\), are equal to those of the uniform matroid of same rank and groundset. This confirms Conjecture \ref{conj:nonnegativity} for all sparse paving matroids, which are conjectured to be predominant among all matroids, see \cite{MNWW}.

We can simplify the nonnegativity question posed in Conjecture \ref{conj:nonnegativity} by making the following observations. In \cite[Theorem 4.1]{JP} the first author showed how to compute the Schubert coefficients of a matroid in terms of its direct summands. This result reduces Conjecture \ref{conj:nonnegativity} to the case of connected matroids. Two special cases that are of interest to us are the cases of adding a loop or a coloop to \(\M\).
\begin{proposition}
    \label{prop:add loop or coloop}
    Let \(\M\) be a matroid of rank \(k\) on \([n]\). Then in \(A^\bullet(G(k,n+1))\)
    \[
        \Sc(\M\oplus \U_{0,1}) \ = \ \Sc(\M) \ \sigma_{[1^k]} \, ,
    \]
    where we abuse notation and think of \(\Sc(\M)\) as an element of \(A^\bullet(G(k,n+1))\).
    Similarly, in \(A^\bullet(G(k+1,n+1))\) we have
    \[
        \Sc(\M\oplus \U_{1,1}) \ = \ \Sc(\M) \ \sigma_{[n-k]} \, .
    \]
\end{proposition}
In other words, the Schubert coefficients of \(\M \oplus \U_{0,1}\) can be obtained from those of~\(\M\) by adding a column of \(k\) boxes to the left of each partition. Similarly, the Schubert coefficients of \(\M \oplus \U_{1,1}\) can be obtained by adding a row of \(n-k\) boxes to the top of each partition. With the dual class we simply  have
\[
    \Sc^c(\M\oplus \U_{0,1}) \ = \ \Sc^c(\M\oplus \U_{1,1}) \ = \ \Sc^c(\M) \, .
\]
Note that these classes are not elements of the same Chow rings; however, since the equalities above hold when we consider them as formal linear combinations of Schubert cycles, we abuse notation and write an equality for the Chow classes as well.

It is also easy to see that the Schubert coefficients behave well with respect to matroid duality. Recall the isomorphism of Grassmannians \(G(k,n) \to G(n-k,n)\) defined by sending a point \(x\) to its orthogonal complement \(x^\bot\). This induces an isomorphism of Chow rings given by sending \(\sigma_\eta\) to the Schubert cycle \(\sigma_{\eta^t}\) of the transposition of \(\eta\). The matroid \(\M_{x^\bot}\) is the dual of \(\M_x\), and the class \(\Sc_m(\M_x)\) is mapped to \(\Sc_m(\M_{x^\bot})\). Hence \(d_{\eta,m}(\M^*) = d_{\eta^t,m}(\M)\).

\subsection{\texorpdfstring{\(K\)-classes}{K-classes}
}

In this section we recap results of \cite{FinkBerget_matrix} and \cite{bergetfink_K} by Berget and Fink that allow us to define the Chow classes of arbitrary matroids. 
The group \(G =  GL_k(\CC) \times (\CC^*)^n\) acts on the affine space of \(k \times n\) matrices \(\mathbb A ^{k \times n}\) via the usual row operations and scaling columns. To a matrix \(v \in \mathbb A^{k\times n}\) one may associate the \emph{matrix orbit closure} \(X_v\), which is the Zariski closure of the orbit of the matrix~\(v\) by the group~\(G\).
Let \(K_0^G(\mathbb A^{k\times n})\) denote the Grothendieck group of \(G\)-equivariant coherent sheaves on \(\mathbb A^{k\times n}\). The group \(K_0^G(\mathbb A^{k\times n})\) can be identified with the group of Laurent polynomials in the variables \(u_1, \dots , u_k, t_1, \dots, t_n\) with integer coefficients,
\[
    K_0^G(\mathbb A^{k\times n}) \ \cong \ \ZZ[u_1^\pm, \dots , u_k^\pm, t_1^\pm, \dots, t_n^\pm] \, .
\]
In \cite{FinkBerget_matrix} the authors studied the class of \(X_v\) in \(K_0^G(\mathbb A^{k\times n})\), and in \cite{bergetfink_K} they found an expression that only depends on the matroid \(\M_v\) represented by the columns of \(v\). This can be used to extend the definition of \(K\)-class to arbitrary, not necessarily representable, matroids. To state this formula we introduce some notation. 

Let \(\M\) be a matroid of rank \(k\) on \([n]\). For a given permutation \(\omega\) of the set~\([n]\) we denote by \(B(\omega,\M)\) the smallest basis of \(\M\) in the lexicographical order induced by \(\omega\). We may write \(B(\omega)\) if the matroid is clear from context. In the ring of rational functions \(\ZZ (u_1,\ldots,u_k,t_1,\ldots,t_n)\) we define the following expressions
\[
    P^n_{B,\M} \ = \ \sum_{\substack{\omega \in \mathfrak{S}_n \\ B = B(\omega,\M)}} \prod_{\ell = 1}^{n-1} \frac{t_{\omega_\ell}}{t_{\omega_\ell} - t_{\omega_{\ell + 1}}} \qquad Q^k_j \ = \ \prod_{i \in [k]} (1 - u_it_j) \, .
\]
%\[
%    p^n_{B,\M} \ = \ \sum_{\substack{\omega \in \mathfrak{S}_n \\ B = B(\omega,\M)}} \prod_{\ell = 1}^{n-1} \frac{1}{t_{\omega_\ell}- t_{\omega_{\ell + 1}}} \qquad q^k_j \ = \ \prod_{i \in [k]} (u_i + t_j) \, .
%\]

\begin{theorem}[{\cite[Theorem 8.1]{bergetfink_K}}]
    \label{thm:K formula}
    Let \(\M\) be a matroid of rank \(k\) on \([n]\). The class of the matroid \(\M\) in \(K_0^G(\mathbb A^{k\times n})\) is given by 
    \[
        \mathcal K(\M) \ = \ \sum_{B \in \mathcal{B}(\M)} P^n_B \prod_{\substack{j \in [n] \\ j \not\in B}}Q^k_j \, .
    \]
    If \(\M\) is represented by the matrix \(v \in \mathbb A^{k\times n}\), then this is the class of \(X_v\) in \(K_0^G(\mathbb A^{k\times n})\). Moreover, \(\mathcal K\) is a valuation.
\end{theorem}
A priori, \(\mathcal K(\M)\) is a rational function but, in fact, \cite[Theorem 8.3]{bergetfink_K} shows that it is a polynomial with integer coefficients in \(u_1,\dots, u_k, t_1, \dots , t_n\), symmetric in the variables \(u_1, \dots, u_k\). As explained in \cite{bergetfink_K}, we can recover the Chow class \(\Sc(\M)\) from \(\mathcal K(\M)\) as follows. Let \(\KtoSc\) be the operator on \(\ZZ[u_1,\dots,u_k, t_1,\dots, t_n]\) defined by first making the substitutions \(u_i \mapsto 1-u_i\), \(t_i \mapsto 1 - t_i\) and then gathering the lowest degree terms. By applying \(\KtoSc\) to \(\mathcal K(\M)\) and evaluating all variables \(t1,\ldots,t_n\) equal to zero, we obtain a homogeneous symmetric polynomial in \(u_1, \dots , u_k\) of degree \(k(n-k) - (n-\kappa(\M))\). Expanding this polynomial in the Schur basis gives
\[
    \KtoSc(\mathcal K(\M)) \big|_{t = 0} \ = \ \sum_\eta d_\eta(\M) s_\eta(u_1,\dots , u_k) \, ,
\]
where we write \(t=0\) as a shorthand notation to mean setting all variables \(t_1,\ldots,t_n\) equal to zero and \(d_\eta(\M)\) agrees with the Schubert coefficient of \(\M\) as defined in Definition \ref{def:Schubert coefficients}. If we identify the Schubert cycle \(\sigma_\eta \in A^\bullet(G(k,n))\) with the Schur polynomial \(s_\eta(u_1,\dots, u_k)\) we may write
\[
    \Sc(\M) \ = \ \KtoSc(\mathcal K(\M)) \big|_{t = 0} \, .
\]

\begin{example}
    \label{ex:U25 from K}
    Theorem \ref{thm:K formula} tells us that the \(K\)-class of the uniform matroid~\(\U_{2,5}\) is
    \begin{align*}
        \mathcal K(\U_{2,5}) \ &= \ 2(u_1^3u_2^2 + u_1^2u_2^3)t_1t_2t_3t_4t_5 \\
        & \quad -  u_1^2u_2^2(t_1t_2t_3t_4 + t_1t_2t_3t_5 + t_1t_2t_4t_5 + t_1t_3t_4t_5 + t_2t_3t_4t_5) + 1 \, .
    \end{align*}
    Applying the operator \(\KtoSc\) results in
    \begin{align*}
        \KtoSc(K(\U_{2,5})) \ &= 3u_1^3 + 4u_1u_2 + 3u_2^2 + 2(u_1+u_2)(t_1 + t_2 + t_3 + t_4 + t_5) \\
    & \quad + t_1t_2 + t_1t_3 + t_1t_4 + t_1t_5 + t_2t_3 + t_2t_4 + t_2t_5 + t_3t_4 + t_3t_5 + t_4t_5 \, .
    \end{align*}
    Now evaluating at \(t = 0\) gives
    \[
        \KtoSc(\mathcal K(\U_{2,5}))\big|_{t = 0} \ = \
    3u_1^3 + 4u_1u_2 + 3u_2^2 \ = \ 
    3s_{\scalebox{0.3}{\young(~~)}}(u_1,u_2) + 1s_{\scalebox{0.3}{\young(~,~)}}(u_1,u_2) \, ,
    \]
    which agrees with \(\Sc(\U_{2,5})\) as computed in Example \ref{ex:U25}.
\end{example}

\section{An identity for Schubert coefficients}\label{sec:series parallel identity}\noindent Our goal in this section is to prove the next theorem, an identity that relates Schubert coefficients of a matroid to those of its series and parallel extensions. Recall that we denote by \(\parallelExt^b(\M)\) the matroid where we add \(b\) elements in parallel to the last element of the ground set of \(\M\) and by \(\seriesExt(\M)\) the series extension of \(\M\) with respect to the last element. Furthermore, let \(\rmv_b(\eta)\) denote the set of partitions that can be obtained from \(\eta\) by removing \(b\) boxes, at most one per column. 
Finally, let \(\rmv_b(s_\eta) = \sum_{\mu \in \rmv_b(\eta)} s_\mu\) which we extend linearly to all symmetric functions.

\begin{theorem}
    \label{thm:main identity}
    Let \(\M\) be a matroid of rank \(k\) on \([n]\) and let \(b \geq 1\). Then
    \[
        \Sc\left(\parallelExt^{b-1}\seriesExt(\M) \oplus \U_{0,1}\right) + \Sc\left(\parallelExt^b(\M) \oplus \U_{1,1}\right) \ = \ \rmv_b\left(\vphantom{\parallelExt^{b}} {\Sc(\M \oplus \U_{1,1} \oplus \U_{0,b})}\right) \, .
    \]
\end{theorem}

In Section~\ref{sec:chow class of snakes} we use Theorem \ref{thm:main identity} to describe the Chow class of snake matroids and give a combinatorial description of the Schubert coefficients of all lattice path matroids. Paired with the findings of \cite{Hampe}, that we have mentioned in
Section~\ref{sec:matroids}, this leads to a new, algorithmic way of computing Schubert coefficients of arbitrary matroids, see Appendix \ref{sec:appendix}.

The remainder of this section is dedicated to the proof of Theorem \ref{thm:main identity}. We decided to split it into  various steps that might be of interest on their own. The first of these steps is an expression for the class \(\mathcal K(\parallelExt(\M))\) of the parallel extension of \(\M\) as studied in~\cite{bergetfink_K}. To formulate our next results we define the following divided difference operators on the polynomial ring \(\ZZ[u_1,\dots, u_k, t_1, \dots, t_n]\). For \(1 \leq i < n\), we define
\[
    \delta_i(f) =  \frac{t_i f - t_{i+1} \tau_i f}{t_i - t_{i+1}}
\]
where \(\tau_i\) is the simple transposition swapping \(t_i\) and \(t_{i+1}\). The following is a slight modification of a result by Berget and Fink, which uses this operator.

\begin{theorem}
    \label{thm:K parext}
    Let \(\M\) be a matroid or rank \(k\) on \([n]\) and let \(\parallelExt(\M)\) be the parallel extension of \(\M\) with respect to the last element \(n\) and let \(\N\) be any matroid. Then
    \[
        \mathcal K(\parallelExt(\M) \oplus \N) \ = \ \delta_n(\mathcal K(\M\oplus \U_{0,1} \oplus \N)) \, .
    \]
\end{theorem}
\begin{proof}
    By \cite[Theorem 9.3]{FinkBerget_matrix}, we have \(\mathcal K(\parallelExt(\M)) = \delta_n(\mathcal K(\M))\) for \(\M\) representable. The result follows by applying this to \(\N\oplus \M\), relabeling the variables \(t_i\) accordingly and by Remark~\ref{rem:valuation from rep}.
\end{proof}
Next we are going to achieve an analogous result for series extensions. Our proof follows the lines of \cite[Theorem~3.2]{FinkBerget_matrix} closely and relies on \cite[Lemma 9.4]{FinkBerget_matrix}, which we thus restate here for completeness.
\begin{lemma}{{\cite[Lemma 9.4]{FinkBerget_matrix}}}
    \label{lem:liegroupsK}
    Let \(T \subseteq B \subseteq P\) be a triple of Lie groups such that \(P/B \cong \mathbb P^1\) and \(T\) is a torus acting with weight \(\mu\) on \(\mathfrak{p}/\mathfrak{b}\). Let \(r \in N_P(T)\) be an element of the normalizer of \(T\) in \(P\), inducing an automorphism \(r\) of the weight lattice \(T^*\) such that \(r \cdot \mu = - \mu\). Let \(V\) be a \(P\)-representation and \(X\subseteq V\) a \(B\)-invariant subvariety. Then in the \(T\)-equivariant \(K\)-theory ring \(K_0^T(V)\) we have the equality of \(K\)-classes
    \[
    d \mathcal K (P\cdot X) \ = \ \frac{\mathcal K(X) -\mu r \mathcal K(X)}{1 - \mu} \, ,
    \]
    where \(d\) is the degree of the map \(P\times^B X \to P \cdot X\).
\end{lemma}

\begin{theorem}
    \label{thm:K serext}
    Let \(\M\) be a matroid of rank \(k\) on \([n]\), \(\seriesExt(\M)\) be its series extension with respect to its last element and \(\N\) be any matroid. Then
    \[
        \mathcal K(\seriesExt(\M) \oplus \N) \ = \ \delta_n \circ \tau_n (\mathcal K(\M\oplus \U_{1,1} \oplus \N)) \, .
    \]
\end{theorem}

\begin{proof}
    As in the proof of Theorem \ref{thm:K parext}, it suffices to show that \(\mathcal K(\seriesExt(\M)) = \delta_n \circ \tau_n (\mathcal K(\M\oplus \U_{1,1}))\) for any representable matroid \(\M\).
       
    Let \(v = (v_1 \ v_2 \dots v_n) \in \mathbb{A}^{(k+1) \times n}\) be a representation of the matroid \(\M\) of rank \(k\). Notice that since \(\M\) is of rank \(k\), the matrix \(v\) is not of full rank. This allows us to pick \(u \not\in  \spn\{v_1, \dots, v_n\}\) such that \(v^c = (v_1 \dots v_n \ u)\) represents \(\M \oplus \U_{1,1}\) and \(v^s = (v_1 \dots v_n + u \ u)\) represents \(\M^s\). Now to apply Lemma \ref{lem:liegroupsK} let 
    \begin{align*}
        V \ &=\ \mathbb{A}^{(k+1)\times(n+1)} \ , \\
        X \ &=\ X_{v^c} \subseteq V \ , \\
        P \ &=\ GL_{k+1} \times (\CC^*)^{n-1} \times GL_2 \ , \\
        B \ &=\ GL_{k+1} \times (\CC^*)^{n-1} \times 
        \left\{ \left[
        \begin{array}{cc}
            * & * \\
            0 & *
        \end{array}
        \right]\right\} \ ,
    \end{align*}
    and let \(T\) be the maximal torus in \(G = GL_{k+1} \times (\CC^*)^{n+1}\).
    
    %Here \(P\) and \(B\) act on \(V\), \(GL_{k+1}\) acts from the left, \((\CC^*)^{n-1}\) acts on the first \(n-1\) columns and \(GL_2\) acts from the right on the last two columns. i.e.
    %\[
    %    \left[
    %    \begin{array}{cc}
    %        a & b \\
    %        c & d
    %    \end{array}
    %    \right]
    %    (v_1 \dots v_{n-1} \ v_n \ v_{n+1}) \ = \ (v_1 \dots v_{n-1} \ av_n + cv_{n+1} \ bv_n + dv_{n+1}) \enspace .
    %\]
    
    Notice that as in the proof of \cite[Theorem 9.3]{FinkBerget_matrix}, the variety \(X\) is \(B\)-equivariant, \(P\cdot X = X_{v^s}\), \(P \times^B X \to P \cdot X\) is degree \(d = 1\) and \(P/B \simeq \mathbb{P}^1\).
    %The isomorphism \(P/B \simeq \mathbb{P}^1\) is given by 
    %\[
    %    \left[
    %    \begin{array}{cc}
    %        a & b \\
    %        c & d
    %    \end{array}
    %    \right]
    %    \longmapsto (a : c).
    %\]
    %T acts by 
    %\[
    %    \left[
    %    \begin{array}{cc}
    %        t_n & 0 \\
    %        0 & t_{n+1}
    %    \end{array}
    %    \right]
    %    \left[
    %    \begin{array}{cc}
    %        a & b \\
    %        c & d
    %    \end{array}
    %    \right]
    %    = \left[
    %    \begin{array}{cc}
    %        t_na & t_nb \\
    %        t_{n+1}c & t_{n+1}d
    %    \end{array}
    %    \right]
    %\]
    %i.e. if \(t \in T\) then \(t(a:c) = \frac{t_n}{t_{n+1}}(a:c)\) so
    The torus \(T\) acts on \(P/B\) with weight \(\mu = \frac{t_n}{t_{n+1}}\). Now let
    \[
        r = \left(I, 1, 
        \left[
        \begin{array}{cc}
            0 & 1 \\
            1 & 0
        \end{array}
        \right]
        \right) \in N_P(T)\, .
    \]
    The automorphism induced by \(r\) is the transposition \(\tau_n\), and clearly \(\tau_n \mu = \mu^{-1}\). By Lemma \ref{lem:liegroupsK}, we then have
    \[
        d \mathcal K(P\cdot X) \ = \ \frac{\mathcal K(X) - \mu r \mathcal K(X)}{1-\mu} \, .
    \]
    The left hand side is \(d\mathcal K(P\cdot X) = \mathcal K(X_{v^s}) = \mathcal K(\seriesExt(\M))\) and the right hand side is
    \[
        \frac{\mathcal K(X) - \mu r \mathcal K(X)}{1-\mu} \ = \ \frac{\mathcal K(\M\oplus \U_{1,1}) - \frac{t_n}{t_{n+1}} \tau_n \mathcal K(\M\oplus \U_{1,1})}{1 - \frac{t_n}{t_{n+1}}} \ = \ \delta_n \circ \tau_n(\mathcal K(\M\oplus \U_{1,1})) \, . \qedhere
    \]
\end{proof}
Our strategy to prove Theorem \ref{thm:main identity} is to use Theorem \ref{thm:K parext} and Theorem \ref{thm:K serext} to obtain an expression for \(\mathcal K\) that specializes to Theorem \ref{thm:main identity} when we apply the operator \(\KtoSc\) and set \(t=0\) as in the previous section. To prepare for the proof we gather a number of lemmas.

\begin{lemma}
    \label{lem:perm lemma}
    For any permutation \(\pi \in \mathfrak S_n\),
    \begin{equation}\label{eq:perm lemma}
        \prod_{\ell = 1}^{n-1} \frac{t_{\pi_\ell}}{t_{\pi_\ell} - t_{\pi_{\ell+1}}} \ = \ \sum_{\substack{\omega \in \mathfrak S_{n+1} \\ \omega \setminus \{n+1\} = \pi}} \prod_{\ell = 1}^n \frac{t_{\omega_\ell}}{t_{\omega_\ell} - t_{\omega_{\ell+1}}} \, .
    \end{equation}
\end{lemma}

\begin{proof}
    By selecting the index \(j\) such that \(\omega_{j+1}=n+1\), the right-hand side of \eqref{eq:perm lemma} equals
    \begin{equation}\label{eq:perm lemma first step}
        \frac{t_{n+1}}{t_{n+1} - t_{\pi_1}} \cdot \prod_{\ell = 1}^{n-1} \frac{t_{\pi_\ell}}{t_{\pi_\ell} - t_{\pi_{\ell+1}}} 
        + 
        \underbrace{\sum_{j = 1}^n \frac{t_{\pi_j}}{t_{\pi_j} - t_{n+1}} \frac{t_{n+1}}{t_{n+1} - t_{\pi_{j+1}}} \cdot \prod_{\ell \neq j} \frac{t_{\pi_\ell}}{t_{\pi_\ell} - t_{\pi_{\ell+1}}}}_{\eqqcolon A}\, ,
    \end{equation}
    where
    \begin{equation}\label{eq:perm lemma second step}
        A = t_{n+1} \cdot \prod_{\ell = 1}^{n-1} \frac{t_{\pi_\ell}}{t_{\pi_\ell} - t_{\pi_{\ell+1}}} \cdot {\underbrace{\sum_{j = 1}^n \frac{t_{\pi_j} - t_{\pi_{j+1}}}{(t_{\pi_j} - t_{n+1})\cdot(t_{n+1} - t_{\pi_{j+1}})}}_{\eqqcolon B}}\, 
    \end{equation}
    and
    \begin{align*}
        B
        &= \sum_{j = 1}^n \frac{(t_{\pi_j} - t_{n+1}) + (t_{n+1} - t_{\pi_{j+1}})}{(t_{\pi_j} - t_{n+1})(t_{n+1} - t_{\pi_{j+1}})} 
        = \sum_{j = 1}^n \left(\frac{1}{t_{n+1} - t_{\pi_{j+1}}} - \frac{1}{t_{n+1} - t_{\pi_j}}\right)\\
        &= -\frac{1}{t_{n+1} - t_{\pi_1}} + \frac{1}{t_{n+1}}
        = \frac{1}{t_{n+1}}\cdot \frac{t_{\pi_1}}{t_{\pi_1} - t_{n+1}}\, .
    \end{align*}
    Therefore \eqref{eq:perm lemma second step} becomes
    \begin{equation*}
        A = \frac{t_{\pi_1}}{t_{\pi_1} - t_{n+1}} \cdot \prod_{\ell = 1}^{n-1} \frac{t_{\pi_\ell}}{t_{\pi_\ell} - t_{\pi_{\ell+1}}}\ ,
    \end{equation*}
    and thus \eqref{eq:perm lemma first step} simplifies to the claimed formula.
\end{proof}
This lets us compute how the \(K\)-class of a matroid changes when adding a loop.
\begin{lemma}
    \label{lem:add loop}
    Let \(\M\) be a matroid of rank \(k\) on \([n]\). Then
    \begin{equation*}
        \mathcal K(\M \oplus \U_{0,1}) \ = \ \mathcal K(\M) \cdot Q^{k}_{n+1} \, .
    \end{equation*}
\end{lemma}

\begin{proof}
    The bases of \(\M \oplus U_{0,1}\) are exactly the bases of \(\M\), and so if \(\pi \in \mathfrak S_n\) and \(\omega \setminus \{n+1\} = \pi\), then the lexicographical first basis of \(\M\) with respect to the order \(\pi\) is the same as the lexicographical first basis of \(\M \oplus \U_{0,1}\) with respect to the order \(\omega\). Lemma~\ref{lem:perm lemma} and the definition of the function \(\mathcal K\) imply that
    \begin{align*}
        \mathcal K(\M \oplus \U_{0,1})  =  \sum_{B \in \mathcal B(\M \oplus \U_{0,1})} P^{n+1}_B \prod_{\substack{j \in [n+1] \\ j \notin B}} Q^k_j 
        = \ Q^k_{n+1}\sum_{B \in \mathcal B(\M)} P^n_B \prod_{\substack{j \in [n] \\ j \notin B}} Q^k_j
        = \ Q^k_{n+1} \mathcal K(\M) \, .
    \end{align*}
\end{proof}

\begin{remark}
    \label{rmk:add cloop}
    Notice that by Lemma \ref{lem:perm lemma} and the fact that a coloop belongs to every basis, the polynomial \(\mathcal K(\M \oplus \U_{1,1})\) does not contain the variable \(t_{n+1}\).
\end{remark}

\begin{lemma}
    \label{lem:Krec}
    Let \(\M\) be a matroid of rank \(k\) on \([n]\) and \(b \geq 1\). Then
    \begin{align*}
        & \mathcal K(\parallelExt^{b-1}\seriesExt(\M)\oplus \U_{0,1}) \ + \ \mathcal K(\parallelExt^b(\M)\oplus \U_{1,1}) \\
        &= \ \delta_{n+b-1}\circ \cdots \circ\delta_{n+1} \left(Q_{n+b}^{k+1} \cdots Q_{n+2}^{k+1}\left[Q_{n+b+1}^{k+1}\mathcal K(\seriesExt(\M)) + \mathcal K(\parallelExt(\M) \oplus \U_{1,1})\right]\right) \, .
    \end{align*}
\end{lemma}
\begin{proof}
    First we apply Theorem \ref{thm:K parext} and Lemma \ref{lem:add loop} \(b-1\) times each to get
    \begin{align*}
        & \mathcal K(\parallelExt^b(\M) \oplus \U_{1,1}) \ = \ \delta_{n+b-1} \circ \cdots \circ \delta_{n+1} \left(\mathcal K(\parallelExt(\M) \oplus \U_{0,b-1}\oplus \U_{1,1})\right)\\
        & = \ \delta_{n+b-1} \circ \cdots \circ \delta_{n+1}\left(Q^{k+1}_{n+b} \cdots  Q^{k+1}_{n+2} \ \mathcal K(\parallelExt(\M) \oplus \U_{1,1})\right) \, .
    \end{align*}
    Similarly we get 
    \[
        \mathcal K(\parallelExt^{b-1}\seriesExt(\M) \oplus \U_{0,1}) \ = \ \delta_{n+b-1} \circ \cdots \circ \delta_{n+1} \left(Q^{k+1}_{n+b+1} \cdots Q^{k+1}_{n+2} \ \mathcal K(\seriesExt(\M))\right) \, .
    \]
    Adding these two expressions leads to the desired result.
\end{proof}

We continue by proving a statement that allows us to deal with the right hand side of Theorem~\ref{thm:main identity} by explicitly computing some Littlewood--Richardson coefficients. Recall that a \emph{lattice word} is a sequence of positive integers where in every intial subword there are at least as many occurrences of \(i\) as there are of \(i+1\). For given partitions \(\lambda\), \(\mu\) and \(\eta\) that satisfy \(|\lambda| = |\mu| + |\eta|\), a \emph{Littlewood--Richardson tableaux} of skew shape \(\lambda/\mu\) and content \(\eta\) is a semistandard skew tableaux such that the sequence obtained by reading its rows from right to left is a lattice word. The following combinatorial interpretation can be found in {\cite[Theorem~A1.3.3]{Stanley-EC2}}.
\begin{theorem}[Littlewood--Richardson rule]
    Let \(\lambda\), \(\mu\) and \(\eta\) be partitions. The Littlewood--Richardson coefficient \(c_{\mu,\eta}^\lambda\) is equal to the number of Littlewood--Richardson tableaux of skew shape \(\lambda/\mu\) and content \(\eta\). 
\end{theorem}
\begin{lemma}
    \label{lem:rmv}
    Let \(\eta \subseteq (k+1) \times (n-k)\) and \(b \geq 1\). In the ring of symmetric polynomials with \(k+1\) variables holds 
    \[
        \rmv_b(s_{[b^{k+1}]} s_\eta) = s_{[b^k]} s_{\eta} \, .
    \]
\end{lemma}

\begin{proof}
    First notice that it is sufficient to only consider partitions with at most \(k+1\) parts as we work in the ring of symmetric polynomials with \(k+1\) variables. In this ring we have \(s_{[b^{k+1}]} s_\eta = (s_{[1^{k+1}]})^b s_\eta = s_{\square\eta}\) where \(\square\eta = [\eta_1 + b, \eta_2 + b, \dots, \eta_{k+1} + b]\), by applying Pieri's formula (Theorem~\ref{thm:pieri}) multiple times. Additonally,
    the Littlewood--Richardson rule shows that
    \[
        s_{[b^k]}s_\eta \ = \ \sum_{\lambda} c^\lambda_{[b^k], \eta} s_\lambda \, ,
    \]
    where \(c^\lambda_{[b^k], \eta}\) is the number of Littlewood--Richardson tableaux of skew shape \(\lambda / [b^k]\) with content \(\eta\).
    Therefore, the claim that we want to prove is that \(c^\lambda_{[b^k],\eta} = 1\) whenever \(\lambda \in \rmv_b(\square \eta)\) and \(0\) otherwise.

    To this end, assume \(T\) is a Littlewood–Richardson tableau of skew shape \(\lambda/[b^k]\) with content~\(\eta\).
    Then the filling of row \(i\) consists of \(\lambda_i-b\) copies of the number \(i\) for all \(i\leq k\), as the subword formed by the first \(i\) rows has to be a lattice word and the content has to increase row by row; in particular, the filling of \(\lambda/[b^k]\) is unique.
    This implies \(b\leq \lambda_i\leq b+\eta_i\) for all \(i\leq k\), and as \(|\lambda|-kb = |\eta|\) we obtain \(\lambda_{k+1}=\eta_{k+1}+\sum_{i=1}^k (\eta_i-\lambda_i+b)\). This last row of the tableau is filled with \(\eta_i-\lambda_i+b\) copies of the number \(i\) for each \(i \leq k\), and \(\eta_{k+1}\) copies of \(k+1\).
    We conclude that \(\lambda_{k+1}\leq \eta_{k+1}+b\), as otherwise the \((b+1)\)st entry on row \(k+1\) is in \([k]\), which would mean that the \((b+1)\)st column of our filling would not be increasing. Hence, \(\lambda \subseteq \square\eta\).
    
    Moreover, in the lattice word corresponding to \(T\), all \(\eta_{i+1}\) appearances of \(i+1\) are read before the \((\lambda_i-b+1)\)st appearances of \(i\), implying that \(\lambda_i-b\geq \eta_{i+1}\). This shows that \(\lambda \in \rmv_b(\square \eta)\). We already saw that \(\lambda\) has a unique filling, so we have completed the proof.
\end{proof}

\begin{example}
    Figure \ref{fig:LR tableaux} shows an example and a nonexample of a Littlewood--Richardson skew tableau as in the proof above. In the nonexample the word, read row by row from right to left, is not a lattice word as all \(3\)'s come before the last \(2\).
\end{example}

\begin{figure}[ht]
    \centering
    \begin{tikzpicture}[scale=0.45]
    \draw[step = 1] (0,0) grid (5,5);
    \draw[step = 1] (5,2) grid (7,5);
    \draw[step = 1] (7,4) grid (8,5);
    \draw[very thick, Red] (0,1.05) -- (3.95,1.05) -- (3.95,5) -- (0,5) -- (0,1);
    \draw[very thick, Blue](4.05,0.95) -- (5,0.95) -- (5,2) -- (6,2) -- (6,3) -- (7,3) -- (7,5) -- (4.05,5) -- (4.05,0.95) -- (0,0.95) -- (0,0) -- (3,0) -- (3,0.95);
    \node at (4.5,4.5) {$1$};
    \node at (5.5,4.5) {$1$};
    \node at (6.5,4.5) {$1$};
    \node at (7.5,4.5) {\textcolor{Red}{\(\times\)}};
    \node at (4.5,3.5) {$2$};
    \node at (5.5,3.5) {$2$};
    \node at (6.5,3.5) {$2$};
    \node at (4.5,2.5) {$3$};
    \node at (5.5,2.5) {$3$};
    \node at (6.5,2.5) {\textcolor{Red}{\(\times\)}};
    \node at (4.5,1.5) {$4$};
    \node at (.5,.5) {$1$};
    \node at (1.5,.5) {$3$};
    \node at (2.5,.5) {$5$};
    \node at (3.5,.5) {\textcolor{Red}{\(\times\)}};
    \node at (4.5,.5) {\textcolor{Red}{\(\times\)}};
\end{tikzpicture}
    \begin{tikzpicture}[scale=0.45]
    \draw[step = 1] (0,0) grid (5,5);
    \draw[step = 1] (5,2) grid (7,5);
    \draw[step = 1] (7,4) grid (8,5);
    \draw[very thick, Red] (0,1.05) -- (3.95,1.05) -- (3.95,5) -- (0,5) -- (0,1);
    \draw[very thick, Blue] (0,0) -- (5,0) -- (5,3) -- (6,3) -- (6,4) -- (7,4) -- (7,5) -- (4.05,5) -- (4.05,0.95) -- (0,0.95) -- (0,0);
    \node at (4.5,4.5) {$1$};
    \node at (5.5,4.5) {$1$};
    \node at (6.5,4.5) {$1$};
    \node at (7.5,4.5) {\textcolor{Red}{\(\times\)}};
    \node at (4.5,3.5) {$2$};
    \node at (5.5,3.5) {$2$};
    \node at (6.5,3.5) {\textcolor{Red}{\(\times\)}};
    \node at (4.5,2.5) {$3$};
    \node at (5.5,2.5) {\textcolor{Red}{\(\times\)}};
    \node at (6.5,2.5) {\textcolor{Red}{\(\times\)}};
    \node at (4.5,1.5) {$4$};
    \node at (.5,.5) {$1$};
    \node at (1.5,.5) {$2$};
    \node at (2.5,.5) {$3$};
    \node at (3.5,.5) {$3$};
    \node at (4.5,.5) {$5$};
\end{tikzpicture}
    \caption{An example on the left and nonexample on the right of a Littlewood--Richardson skew tableau with content \(\eta = [4,3^2,1^2]\) and \(b = k = 4\).}
    \label{fig:LR tableaux}
\end{figure}

\noindent Now consider the polynomials
\[
    q_j^k \ \coloneqq \ \KtoSc(Q_j^k) \ = \ \prod_{i \in [k]} (u_i + t_j) \ = \ \sum_{\ell = 0}^{k}e_{k-\ell}(u_1,\dots ,u_k) t_j^\ell \, ,
\]
where \(e_{k-\ell}(u_1,\dots ,u_k)\) is an elementary symmetric polynomial. Consider also the divided difference operators \(\partial_i\) on \(\ZZ[u_1,\dots ,u_k, t_1,\dots, t_n]\) defined by \(\partial_i(f) = \frac{f - \tau_if}{t_i - t_{i+1}}\).

\begin{lemma}
    \label{lem:the last step}
    Fix a positive integer \(k\). For any \(b \geq 1\) we have
    \[
        (-1)^b\partial_b \circ \dots \circ \partial_2 \circ \partial_1 \left( q^{k+1}_{b+1} \dots q^{k+1}_3 q^{k+1}_2 \right) \bigg|_{t = 0}   = \  s_{[b^k]} \, .
    \]
\end{lemma}
\begin{proof}
    Let \(\Sigma_b\) be the polynomial in \(u_1,\dots , u_{k+1}, t_{b+1}\) defined via 
    \[
        \Sigma_b  \ = \ (-1)^b\partial_b \circ \dots \circ \partial_2 \circ \partial_1 \left( q^{k+1}_{b+1} \dots q^{k+1}_3 q^{k+1}_2 \right) \bigg|_{t_1 =\cdots= t_b = 0} \, .
    \]
    We have to show that \(\Sigma_b |_{t_{b+1} = 0} = s_{[b^k]}\). For this sake we claim that 
    \begin{equation}
        \label{eq:sigma poly 1}
        \Sigma_b \ = \ \sum_{\ell = 0}^k \det A(b,\ell)\, t_{b+1}^\ell \, ,
    \end{equation}
    where \(A(b,\ell)\) is the \(b \times b\)-matrix of elementary symmetric polynomials 
    \[
        A(b,\ell)  \ = \  
        \left[
        \begin{array}{ccccc}
            e_k & e_{k+1} & 0 & \dots & 0 \\
            e_{k-1} & e_k & e_{k+1} & \dots & 0\\
            \vdots & \vdots  & \ddots & \ddots & \vdots\\
            e_{k-b+2} & e_{k-b+1} & \dots & e_k & e_{k+1}\\
            e_{k-b+1-\ell} & e_{k-b+2-\ell} & \dots & e_{k-1-\ell} & e_{k-\ell}
        \end{array}
        \right] \, ,
    \]
    where we write \(e_{j}\) for the elementary symmetric polynomial \(e_{j}(u_1,\dots, u_{k+1})\) and use the convention that \(e_0 = 1\) and \(e_j = 0\) whenever \(j\) is 
    negative.
    Here we want to point out that \(e_{k+j}=0\) if \(j>1\) as it is a square-free polynomial in \(k+1\) variables.
    Assuming the equality in \eqref{eq:sigma poly 1} we have that
    \[
        \Sigma_b |_{t_{b+1} = 0} \ = \ \det A(b,0) \ = \ s_{[b^k]} \, ,
    \]
    where the last equality is the Jacobi--Trudi identity, see \cite[Corollary~7.16.2]{Stanley-EC2}.
    We are left to prove \eqref{eq:sigma poly 1}. We begin by manipulating \(\Sigma_b\). First we use that the polynomial \(q^{k+1}_j\) commutes with the divided difference operator \(\partial_i\) whenever \(j>i+1\), i.e., \(\partial_i(q^{k+1}_j f) =~q^{k+1}_j\partial_i(f)\). This way we obtain
    \begin{equation*}
        \label{eq:sigma poly 2}
        \Sigma_b  \ = \ (-1)^b \partial_b( q^{k+1}_{b+1} \ \partial_{b-1}( q^{k+1}_{b} \cdots \partial_2 (q^{k+1}_3 \ \partial_1(q^{k+1}_2)))) \big|_{t = 0} \, .
    \end{equation*}
    Next, we use that \(\partial_j(t_i f) = t_i\, \partial_j(f)\) and that \(q^{k+1}_j\big|_{t_i=0} = q^{k+1}_j\) whenever \(i<j\).
    Hence \(\partial_j(q_{j+1}^{k+1} f)\big|_{t_i=0} = \partial_j(q_{j+1}^{k+1} f\big|_{t_i=0})
    \) for \(i<j\). Applying this step iteratively shows that
    \[
    \Sigma_b\ =\ - \partial_{b}(q^{k+1}_{b+1} \Sigma_{b-1}) \big|_{t_b = 0} \, .
    \] We are now prepared to prove \eqref{eq:sigma poly 1} by induction on \(b\). For the base case \(b = 1\) we have
    \begin{align*}
        \Sigma_1 \ &= \ -\partial_1(q^{k+1}_2)\big|_{t_1 = 0}  \ = \  %\partial_1(q^{k+1}_1)\big|_{t_1 = 0} \ = \ 
        \frac{q^{k+1}_1 - q^{k+1}_2}{t_1 - t_2} \bigg|_{t_1 = 0}
        = \ \sum_{\ell = 0}^{k+1} e_{k+1-\ell} \frac{t_1^\ell-t_2^\ell}{t_1-t_2}\bigg|_{t_1 = 0}\\ 
        %&= \ \sum_{\ell = 1}^{k+1} e_{k+1-\ell}t_2^{\ell-1}\\
        &= \ \sum_{\ell = 0}^k e_{k-\ell} \, t_2^\ell \ = \ \sum_{\ell = 0}^k \det A(1,\ell)\, t_2^\ell \, .
    \end{align*}
    Now let \(b > 1\) and assume by induction that
    \[
        \Sigma_{b-1} \ = \ \sum_{\ell = 0}^{k} \det A(b-1,\ell)\, t_b^\ell \, .
    \]
    With these assumptions we have
    \begin{align*}
        -\partial_b(q^{k+1}_{b+1} \Sigma_{b-1})\big|_{t_b = 0} 
       \ &= \ -\sum_{\ell = 0}^k\det A(b-1,\ell) \ \partial_b(q^{k+1}_{b+1} \, t_{b}^\ell) \big|_{t_b = 0}\\
        &= \ - \sum_{\ell = 0}^k\det A(b-1,\ell) \ \frac{q^{k+1}_{b+1} t_{b}^\ell - q^{k+1}_{b}t^\ell_{b+1}}{t_b-t_{b+1}} \bigg|_{t_b = 0}\\
        &= \ \sum_{\ell = 0}^k\det A(b-1,\ell) \sum_{r = 0}^{k+1}e_{k+1-r} \ \frac{t^r_bt^\ell_{b+1} - t^\ell_bt^r_{b+1}}{t_b - t_{b+1}}\bigg|_{t_b = 0} \, .
    \end{align*}
   We see that in the inner sum we have a geometric sum which equals
    \[
    \frac{t^r_bt^\ell_{b+1} - t^\ell_bt^r_{b+1}}{t_b - t_{b+1}}\bigg|_{t_b = 0} \ =\quad 
    \begin{cases}
        - t^{\ell-1}_{b+1} & \text{if } r = 0 < \ell\\
        t^{r-1}_{b+1} & \text{if } r>\ell=0\\
        0 & \text{otherwise} \, .\\
    \end{cases}
    \]
    Leading us to the following equality
    \begin{align*}
        \Sigma_b &= -\partial_b(q^{k+1}_{b+1} \Sigma_{b-1})\big|_{t_b = 0}\\ &= 
        \det A(b-1,0) \sum_{r=1}^{k+1} e_{k+1-r}\, t^{r-1}_{b+1} 
        - \sum_{\ell=1}^{k} \det A(b-1,\ell) e_{k+1}\, t^{\ell-1}_{b+1}\\
        &= \sum_{\ell = 0}^k \bigl(\det A(b-1,0)\, e_{k-\ell} - \det A(b-1,\ell+1)\, e_{k+1}\bigr)t_{b+1}^\ell
         \ =\  \sum_{\ell = 0}^k \det A(b,\ell) t_{b+1}^\ell \, ,
        \end{align*}
        where we used that \(\det A(b-1,k+1)=0\), as the last row of the matrix vanishes, and in the last equation we applied the identity
        \begin{equation*}
            \det(A(b,\ell)) \ = \ \det(A(b-1,0))\,e_{k-\ell} - \det(A(b-1,\ell+1))\, e_{k+1}
        \end{equation*}
        which we obtain from expanding the determinant with respect to the last column \(b\) of the matrix \(A(b,\ell)\).
        This completes the induction and hence the entire proof.
\end{proof}

We are now prepared to combine  these results to prove  Theorem \ref{thm:main identity}.
\begin{proof}[Proof of Theorem~\ref{thm:main identity}]
    Observe that
    \begin{align*}
        \mathcal K(\parallelExt(\M) \oplus \U_{1,1})\ & = \ \delta_n(\mathcal K(\M\oplus \U_{0,1} \oplus \U_{1,1}))\\
        & =\ \frac{t_nQ_{n+1}^{k+1} \mathcal K(\M \oplus \U_{1,1})}{t_n - t_{n+1}}\ -\ \frac{t_{n+1}Q_{n}^{k+1} \tau_n \mathcal K(\M \oplus \U_{1,1})}{t_n - t_{n+1}}\\
        & \quad+\ \frac{t_nQ_{n}^{k+1} \mathcal K(\M \oplus \U_{1,1})}{t_n - t_{n+1}}\ -\ \frac{t_nQ_{n}^{k+1} \mathcal K(\M \oplus \U_{1,1})}{t_n - t_{n+1}}\\
        &=\ Q_n^{k+1}\delta_n(\mathcal K (\M \oplus \U_{1,1}))\ +\ t_n \frac{Q_{n+1}^{k+1} - Q_n^{k+1}}{t_n - t_{n+1}}\mathcal K (\M \oplus \U_{1,1}) \, .
    \end{align*}
    Here we used Theorem \ref{thm:K parext}, add and subtract the same term and gather like terms. Using Theorem \ref{thm:K serext}, the innermost part of the right hand side of Lemma \ref{lem:Krec} is
    \begin{align*}
        Q_{n+b+1}^{k+1} \mathcal K(\seriesExt(\M)) + \mathcal K(\parallelExt(\M) \oplus \U_{1,1})\ &=\ t_n \frac{Q_{n+1}^{k+1} - Q_n^{k+1}}{t_n - t_{n+1}}\mathcal K (\M \oplus \U_{1,1}) \\
        &\quad+ \ Q_{n+b+1}^{k+1} \ \delta_n\circ \tau_n(\mathcal K(\M\oplus \U_{1,1})) \\
        &\quad+ \ Q_n^{k+1} \ \delta_n(\mathcal K (\M \oplus \U_{1,1}))\ .
    \end{align*}
    Notice that \(\KtoSc(\delta_i(f)) = - \partial_i(\KtoSc(f))\) and \(\partial_n \circ \tau_n(f) = - \partial_n(f) \), so applying \(\KtoSc\) to the equation above gives
    \begin{align*}
        \KtoSc\left(Q_{n+b+1}^{k+1}\mathcal K(\seriesExt(\M)) + \mathcal K(\parallelExt(\M) \oplus \U_{1,1})\right) &= (q_{n+b+1}^{k+1} - q_n^{k+1}) \ \partial_n\left(\KtoSc(\mathcal K(\M\oplus \U_{1,1}))\right)\\
        &\quad+ \partial_n(q_n^{k+1}) \ \KtoSc(\mathcal K(\M \oplus \U_{1,1})) \, .
    \end{align*}
    We now apply \(\KtoSc\) to Lemma \ref{lem:Krec} to obtain
    \begin{align*}
        & \KtoSc(\mathcal K(\parallelExt^{b-1}\seriesExt(\M)\oplus \U_{0,1})) + \KtoSc(\mathcal K(\parallelExt^b(\M)\oplus \U_{1,1}))\\
        %& = (-1)^{b-1} \partial_{n+b-1}\circ \cdots \circ \partial_{n+1}\left(q_{n+b}^{k+1}\cdots q_{n+2}^{k+1}\left[\KtoSc(Q_{n+b+1}^{k+1}\mathcal K(\seriesExt(\M)) + \mathcal K(\parallelExt(\M) \oplus \U_{1,1}))\right]\right)
        & = (-1)^{b-1} \partial_{n+b-1}\circ \cdots \circ \partial_{n+1}\left(q_{n+b}^{k+1}\cdots q_{n+2}^{k+1}\left[(q_{n+b+1}^{k+1} - q_n^{k+1})\ \partial_n\left(\KtoSc(\mathcal K(\M\oplus \U_{1,1}))\right)\right]\right)\\
        &\quad + (-1)^{b-1} \partial_{n+b-1}\circ \cdots \circ \partial_{n+1}\left(q_{n+b}^{k+1} \cdots q_{n+2}^{k+1} \ \partial_n(q_n^{k+1}) \KtoSc(\mathcal K(\M \oplus \U_{1,1}))\right) \, .
    \end{align*}
    Notice that when we evaluate at \(t = 0\) the first term above vanishes since the factor \(q_{n+b+1}^{k+1} - q_n^{k+1}\) can be pulled out of all the divided difference operators. By Remark \ref{rmk:add cloop} we can also pull \(\KtoSc(\mathcal K(\M \oplus \U_{1,1}))\) out of all the divided difference operators. When evaluating at \(t = 0\) the whole expression simplifies to 
    \begin{align*}
        & \Sc(\parallelExt^{b-1}\seriesExt(\M)\oplus \U_{0,1}) + \Sc(\parallelExt^b(\M)\oplus \U_{1,1})\\
        & = \ \Sc(\M \oplus \U_{1,1}) \cdot (-1)^b\partial_{n+b-1}\circ \cdots \circ \partial_{n+1}\circ \partial_n \left(q_{n+b}^{k+1} \cdots q_{n+2}^{k+1}q_{n+1}^{k+1}\right) \bigg |_{t = 0}\, ,
    \end{align*}
    where we used that \(\partial_n(q_{n}^{k+1}) = -\partial_n(q_{n+1}^{k+1})\) and identified the Schubert cycle \(\sigma_\eta\) with the Schur polynomial \(s_\eta = s_\eta(u_1,\dots, u_{k+1})\). By Lemma~\ref{lem:the last step} we get
    \[
        \Sc(\parallelExt^{b-1}\seriesExt(\M)\oplus \U_{0,1}) + \Sc(\parallelExt^b(\M)\oplus \U_{1,1}) \ = \ \Sc(\M \oplus \U_{1,1}) \ s_{[b^k]}
    \]
    and by Proposition \ref{prop:add loop or coloop} we have that
    \[
        \Sc(\M \oplus \U_{1,1} \oplus \U_{0,b}) \ = \ s_{[b^{k+1}]}\Sc(\M \oplus \U_{1,1}) \, .
    \]
    The statement now follows from Lemma \ref{lem:rmv}.
\end{proof}

\section{The Chow class of snake matroids}\label{sec:chow class of snakes}\noindent
In this section, our aim is to prove that the dual Chow class of a snake matroid is a ribbon Schur function. Recall that \(R = R(k,n)\) denotes the quotient of the ring of symmetric functions by the ideal \(\langle s_\lambda \mid \lambda \not \subseteq k\times (n-k)\rangle\) and is isomorphic to the Chow ring \(A^\bullet(G(k,n))\) as mentioned in Section \ref{subsec:schubert coefficients}. In the remainder of the section we do computations in \(R\). In particular, we write \(\Sc^c(\M) = \sum_\eta d_{\eta^c}(\M)s_\eta\).
\begin{theorem}\label{thm:ribbon snake}
    Let \(\S\) be a snake matroid of rank \(k\) on \([n]\) given by the ribbon \(\ribbon(\mathbf b)\). Then the dual Chow class of \(\S\) in \(R\) is equal to 
    \[
        \Sc^c(\S) \ = \ s_{\ribbon(\mathbf b)} \, ,
    \]
    where \(s_{\ribbon(\mathbf b)}\) is the ribbon Schur function of shape \(\ribbon(\mathbf b)\) .
\end{theorem}

As explained in Remark~\ref{rem:snake series parallel}, snake matroids are exactly the matroids that can be built recursively by series and parallel extensions on the last element starting with the uniform matroid~\(\U_{1,2}\). In this case, Theorem \ref{thm:main identity} specializes to an identity involving three different snake matroids.

\begin{corollary}\label{cor:main identity snakes}
    Let \(\mathbf b = (b_1,\ldots, b_k, b_{k+1})\).
    For the snake matroid \(\S(\mathbf b)\) of rank \(k + 1\)
    \[
        \Sc(\S(\mathbf b) \oplus \U_{0,1}) \ = \ \rmv_{b_k}\left(\Sc(\S(\mathbf b ') \oplus \U_{1,1} \oplus \U_{0,b_k})\right) - \Sc(\S(\mathbf b'') \oplus \U_{1,1})\, ,
    \]
    where \(\mathbf b' = (b_1,\ldots, b_k)\) and \(\mathbf b'' = (b_1,\ldots, b_{k-1}, b_k + b_{k+1})\).
    Equivalently,
    \[
        \Sc^c(\S(\mathbf b))\ \ = \ s_{[b_k]}\Sc^c(\S(\mathbf b')) - \Sc^c(\S(\mathbf b'')) \, .
    \]
\end{corollary}

\begin{example}
    \label{ex:snake recursion comp}
    Corollary \ref{cor:main identity snakes} allows us to compute the Schubert coefficients of the snake matroid \(\S(2,1,2,3)\) under the assumption that we know that 
    \[
        \Sc^c(2,1,2) \ = \ s_{[2,2,1]} + s_{[3,1,1]} \ \text { and } \  \Sc^c(2,1,5) \ = \ s_{[5,2,1]} + s_{[6,1,1]} \, .
    \]
    Applying Corollary \ref{cor:main identity snakes} together with Pieri's formula (Theorem~\ref{thm:pieri}) we obtain
    \begin{align*}
        & \Sc^c(\S(2,1,2,3)) \ = \ s_{[3]}\Sc^c(\S(2,1,2)) - \Sc^c(\S(2,1,5)) \\
        &= \ s_{[3]}(s_{[2,2,1]} + s_{[3,1,1]}) - (s_{[5,2,1]} + s_{[6,1,1]}) \\
        & = \ s_{[5,2,1]} + s_{[5,1,1,1]} + s_{[4,3,1]} + s_{[4,2,2]} + 2 s_{[4,2,1,1]} + s_{[3,3,1,1]} + s_{[3,2,2,1]} \, .
    \end{align*}
\end{example}
 Let \(\schurmatrix\) denote the upper triangular infinite matrix over \(R\) defined by \(\schurmatrix=[s_{[j-i]}]_{0 \leq i,j}\). For any pair of \(k\) row and column indices \(\mathbf{r} = \{r_1,\dots,r_k\}\) and \(\mathbf{c} = \{c_1,\dots,c_k\}\) let \(\schurmatrix(\mathbf{r}, \mathbf{c})\) denote the minor of \(\schurmatrix\) associated to \(\mathbf{r}\) and \(\mathbf{c}\). That is, \(\schurmatrix(\mathbf{r}, \mathbf{c}) = \det([s_{[c_j-r_i]}]_{1 \leq i,j \leq k})\). With this notation we can recover the Chow class of a snake matroid as such a determinant.
\begin{definition}
  \label{def:comb formula}
  For a sequence \(\mathbf b = (b_1, \ldots, b_k)\) of nonnegative integers the \emph{descent set associated to \(\mathbf b\)} is
  \[
  \Des(\mathbf b) = \ \SetOf{\sum_{i=1}^s b_i}{ 1 \leq s \leq k-1} \, .
  \]
  We also say that this is the descent set associated to the snake matroid \(\S(\mathbf b)\).
\end{definition}
\begin{lemma}
    \label{lem:snake determinant}
    Let \(\S = \S(\mathbf{b})\) be a snake matroid of rank \(k\) on \([n]\) and let \(\mathcal D~=~\Des(\mathbf b)\). Then we have
    \[
        \schurmatrix(\{0\} \cup \mathcal D, \mathcal D \cup \{n-1\}) \ = \ \sum_{\eta \vdash n-1} d_{\eta^c}(\S) s_\eta \ = \ \Sc^c(\S)\, .
    \]
\end{lemma}
\begin{proof}
    The proof is by induction on the rank, with the case \(k=1\) being trivial. By expanding along the first column we obtain exactly the recursion in Corollary \ref{cor:main identity snakes}, from which we conclude.
\end{proof}

\begin{example}
    One can easily check that the values in Example \ref{ex:snake recursion comp} agree with
    \begin{align*}
    \Sc^c(2,1,2,3) \ &= \ \det \begin{pmatrix}
        s_{[2]} & s_{[3]} & s_{[5]} & s_{[8]} \\
        1 & s_{[1]} & s_{[3]} & s_{[6]} \\
        0 & 1 & s_{[2]} & s_{[5]}\\
        0 & 0 & 1 & s_{[3]}\\
    \end{pmatrix}\\
    &\\
    &= \ s_{[3]} \det \begin{pmatrix}
        s_{[2]} & s_{[3]} & s_{[5]}  \\
        1 & s_{[1]} & s_{[3]} \\
        0 & 1 & s_{[2]}\\
    \end{pmatrix}
    - 1 \cdot \det \begin{pmatrix}
        s_{[2]} & s_{[3]} & s_{[8]}  \\
        1 & s_{[1]} & s_{[6]} \\
        0 & 1 & s_{[5]} \\
    \end{pmatrix} \\
    &\\
    &= \ s_{[3]} \Sc^c(\S(2,1,2)) - \Sc^c(\S(2,1,5)) \, .
    \end{align*}
\end{example}

We are now ready to prove Theorem \ref{thm:ribbon snake}.
\begin{proof}[Proof of Theorem \ref{thm:ribbon snake}]
    Let \(\S = \S(\mathbf b)\) be the snake matroid with associated ribbon \(\ribbon(\mathbf b)\). By inverting the order of the entries of \(\mathbf b\) we obtain an isomorphic snake matroid \(\S' = \S(\mathbf b')\) with \(\mathbf b' = (b_1', \dots , b_k')\) and associated ribbon \(\ribbon(\mathbf b')\). Let \(\Des(\mathbf b') = \{d'_1, \ldots,d'_{k-1} \}\) be the descent set associated to \(\S'\) and set \(d'_0 = 0\) and \(d'_k = n-1\). Then by Lemma \ref{lem:snake determinant} we have
    \[
        \Sc^c(\S) \ = \ \Sc^c(\S') \ = \ \det\left(\left[s_{[d'_j-d'_{i-1}]}\right]_{i,j = 1}^k\right) \, .
    \]
    The ribbon \(\ribbon(\mathbf b) = \lambda/\mu\) satisfies
    %\(\lambda_i = n-k+i-1- \sum_{m = 1}^{i-1}b_m'\), \ \(\mu_j = n-k+j-1-\sum_{m = 1}^{j}b_m'\) and \(d'_j = \sum_{m = 1}^jb_m'\)
    \[
        d'_j - d'_{i-1} \ = \ \lambda_i - \mu_j - i + j \qquad 1 \leq i,j \leq k \, .
    \]
    Now the result follows from the Jacobi--Trudi identity, see for example \cite[Theorem.16.1]{Stanley-EC2}. 
\end{proof}
The following is an immediate consequence of expanding a skew Schur function in the Schur basis.
\begin{corollary}\label{cor:Sc_coefficients_snakes_are_LR_ribbons}
    Let \(\S = \S(\mathbf b)\) be a snake matroid of rank \(k\) on \([n]\) with associated ribbon \(\ribbon(\mathbf b) = \lambda/\mu\). Then
    \[
    d_{\eta^c}(\S) \ = \ c^{\lambda}_{\mu,\eta} \, .
    \]
\end{corollary}
By expanding the determinant in Lemma \ref{lem:snake determinant} we get an alternative formula for the Chow class of snake matroids. Let \(\mathbf{b} = (b_1,b_2,\dots, b_k)\) be a composition of \(n-1\) with \(k\) parts. For any subset \(A = \{i_1,i_2, \dots, i_m\} \subset [k-1]\) let \(\mathbf{b}(A)\) denote the composition of \(n-1\) with \(|A| + 1\) parts given by 
\[
    \mathbf{b}(A) \ = \ (b_1 + \dots + b_{i_1}, b_{i_1+1} + \dots + b_{i_2}, \dots, b_{i_m+1} + \dots + b_k) \, .
\]
\begin{corollary}
    \label{cor:snake alternating formula}
    With notation as above let \(\S = \S(\mathbf{b})\) be a snake matroid of rank \(k\) on~\([n]\). Then 
    \[
        \Sc^c(\S) \ = \ \sum_{A \subseteq [k-1]} (-1)^{k-1-|A|} \prod_{j \in \mathbf{b}(A)} s_{[j]} \, .
    \]
\end{corollary}
When expanding the products of Schur functions in the expression above in the Schur basis, the coefficients are the well known \emph{Kostka numbers}, see \cite[Corollary 7.12.4]{Stanley-EC2}. That is, if \(\mathbf{b}\) is a composition of \(n-1\), then 
\[
    \prod_{j \in \mathbf{b}}s_{[j]} \ = \ \sum_{\eta \vdash n-1} K_{\eta, \mathbf{b}} s_\eta \, ,
\]
\begin{corollary}
    \label{cor:snake Kostka}
    Let \(\S = \S(\mathbf{b})\) be a snake matroid of rank \(k\) on \([n]\) and let \(\eta\subseteq k \times (n-k)\) be a partition of \(n-1\). Then 
    \[
        d_{\eta^c}(\S) \ = \sum_{A\subseteq [k-1]} (-1)^{k-1-|A|} K_{\eta, \mathbf{b}(A)} \, .
    \]
\end{corollary}

Corollary \ref{cor:snake Kostka} greatly simplifies for specific shapes of the partition \(\eta\). 
\begin{corollary}\label{cor:Kostka}
    Let \(\S(\mathbf b)\) be a snake matroid of rank \(k\) on \([n]\) and let \(\eta \vdash n-1\) be a partition of length \(k\). Then 
    \[
    d_{\eta^c}(\S) = K_{\eta, \mathbf b} \,.
    \]
    Similarly, if \(\eta\) is a partition with \(\eta_1 = n-k\), then
    \[
    d_{\eta^c}(\S) = K_{\eta,\mathbf b^*} \, ,
    \]
    where \(b^*\) is the composition associated to the transpose of the ribbon, i.e., \(\ribbon(\mathbf b^*) = \ribbon(\mathbf b)^t \). 
\end{corollary}
\begin{proof}
    It is easy to see that \(K_{\eta,\mathbf b}\) is zero if \(\ell(\mathbf b) < \ell(\eta)\), as there cannot be a legal filling of a semistandard Young tableau with less entries than the number of rows. The only summand in Corollary \ref{cor:snake Kostka} that can contribute when \(\ell(\eta)=k\) is then \(K_{\eta,\mathbf b}\). The result then follows. The result for partitions with \(\eta_1 = n-k\) follow by duality.
\end{proof}

In fact any Kostka number can be obtained as a Schubert coefficient of a snake matroid. Let \(\eta\) be a partition and \(\mathbf{b}\) a composition. For any \(k \geq \max\{\ell(\eta), \ell(\mathbf{b})\}\) let \(\eta' = [\eta_1 +1, \dots , \eta_k + 1]\) and \(\mathbf{b}' = (b_1+1 , \dots , b_k + 1)\), then by Corollary \ref{cor:Kostka} we have 
\[
    K_{\eta, \mathbf{b}} \ = \ K_{\eta', \mathbf{b}'} \ = \ d_{\eta'}(\S(\mathbf{b}')) \, . 
\]

\subsection{A descent statistic interpretation} \label{sec:combinatorial}
When focusing on the Chow class of snake matroids, Corollary \ref{cor:Sc_coefficients_snakes_are_LR_ribbons} tells us that we deal with the well-studied Littlewood--Richardson coefficients as they appear in the Schur expansion of ribbon Schur functions; these were first studied by MacMahon in \cite{macmahon_combinatory_analysis}. 
We follow Gessel who described them in the following combinatorial fashion.
A \emph{descent} in a standard Young tableau is an entry \(i\) such that the entry \(i+1\) is in a row below \(i\). For a standard Young tableau \(T\) we denote the set of descents of \(T\) by \(\Des(T)\), and the number of descents by \(\des(T) = |\Des(T)|\). For a given partition~\(\eta\) of \(\sum_i b_i\), let \(\SYT_{\eta}(\mathbf b)\) denote the set of standard Young tableaux~\(T\) of shape \(\eta\) with
\begin{equation} \label{eq:des condition}
    \Des(T) \ = \ \Des(\mathbf b) \, .
\end{equation}
Gessel describes the Littlewood--Richardson coefficients \(c^\lambda_{\mu,\eta}\) when \(\lambda/\mu\) is a ribbon in terms of standard Young tableaux with a predetermined descent set. The following theorem is a consequence of Theorem \ref{thm:ribbon snake} by applying this description in \cite[Theorem~7]{gessel}.

\begin{theorem} \label{thm:combinatorial Sc snakes}
    Let \(\eta\) be a partition of \(n-1\) and \(\S = \S(\mathbf b)\) a snake matroid of rank \(k\) on the ground set \([n]\).
    The Schubert coefficient of \(\S\) associated to \(\eta^c\) is given by 
      \begin{equation*}
            d_{\eta^c}(\S(\mathbf b)) \ = \ |\SYT_\eta(\mathbf b)| \, .
      \end{equation*}
\end{theorem}
In the rest of this section we give an independent proof of Theorem \ref{thm:combinatorial Sc snakes} based on the recursion in Corollary \ref{cor:main identity snakes}.  This can be viewed together with Theorem \ref{thm:ribbon snake} as a new proof of \cite[Theorem~7]{gessel}.

The standard Young tableaux with descent set given by \eqref{eq:des condition} can be constructed in the following way. First place entries \(1, \ldots, b_1\) in the first row of the tableau. Then, iteratively for each \(s = 2,\ldots, k\), add the \(b_s\) new entries \(\sum_{i=1}^{s-1} b_i + 1, \ldots, \sum_{i=1}^s b_i\) so that there is at most one new entry per column, the new entries are ordered from left to right, each new entry is at the bottom of its column, and the leftmost new entry \(\sum_{i=1}^{s-1} b_i + 1\) is below the previous largest entry \(\sum_{i=1}^{s-1} b_i\).

\begin{example}\label{ex:snake and its tableaux}
    Consider the snake matroid \(\S = \S(2,1,2,3)\). The standard Young tableaux with descent set given by \eqref{eq:des condition} are illustrated in Figure \ref{fig:snake tableaux}, colored to visualize the construction process described above. By Theorem~\ref{thm:combinatorial Sc snakes} it follows that
    \begin{equation*}
        \Sc^c(\S) \ = \ \sigma_{\scalebox{0.3}{\young(~~~~~,~~,~)}} + \sigma_{\scalebox{0.3}{\young(~~~~~,~,~,~)}} + \sigma_{\scalebox{0.3}{\young(~~~~,~~~,~)}} + \sigma_{\scalebox{0.3}{\young(~~~~,~~,~~)}} + 2 \sigma_{\scalebox{0.3}{\young(~~~~,~~,~,~)}} + \sigma_{\scalebox{0.3}{\young(~~~,~~~,~,~)}} + \sigma_{\scalebox{0.3}{\young(~~~,~~,~~,~)}}\; ,
    \end{equation*}
    as we saw in Example \ref{ex:snake recursion comp}.
\end{example}

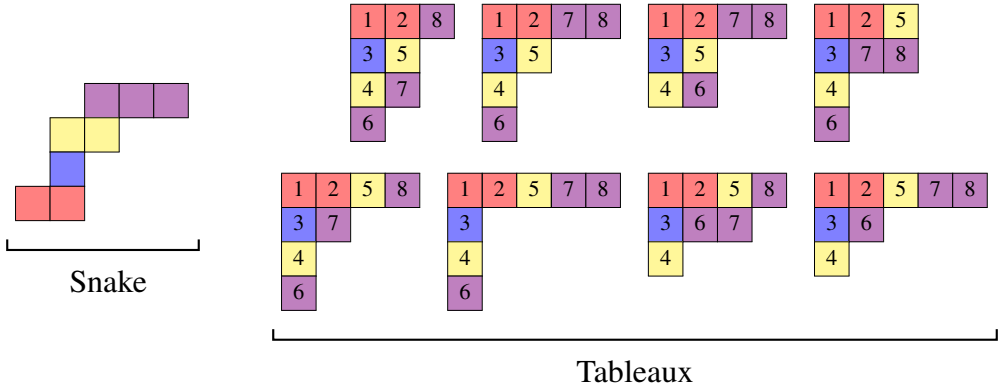
\begin{figure}[ht]
    \centering
    \begin{tikzpicture}[scale=0.7]
\node[anchor=south west, scale=0.7] at (0,1.7) {
    \begin{ytableau}
    \none & \none   & *(Violet)  & *(Violet) & *(Violet)\\
    \none & *(Yellow) & *(Yellow) & \none & \none\\
    \none & *(Blue) & \none & \none & \none\\
    *(Red)& *(Red)& \none & \none & \none\\
    \end{ytableau}
};
\node at (1.9,0.7) {Snake};
\draw[thick] (0,1.5) -- (0,1.3) -- (3.6,1.3) -- (3.6,1.5);

\coordinate (shift) at (5,0);
\begin{scope}[shift=(shift)]
\node[anchor=south west, scale=0.7] at (1.3,3.2) {
    \begin{ytableau}
    *(Red) 1 & *(Red) 2 & *(Violet) 8\\
    *(Blue) 3  & *(Yellow) 5\\
    *(Yellow) 4 & *(Violet) 7\\
    *(Violet) 6 \\
    \end{ytableau}

    \hspace{3mm}
    \begin{ytableau}
    *(Red) 1 & *(Red) 2 & *(Violet) 7 & *(Violet) 8\\
    *(Blue) 3  & *(Yellow) 5\\
    *(Yellow) 4 \\
    *(Violet) 6 \\
    \end{ytableau}

    \hspace{3mm}
    \begin{ytableau}
    *(Red) 1 & *(Red) 2 & *(Violet) 7 & *(Violet) 8\\
    *(Blue) 3  & *(Yellow) 5\\
    *(Yellow) 4 & *(Violet) 6 \\
    \end{ytableau}

    \hspace{3mm}
    \begin{ytableau}
    *(Red) 1 & *(Red) 2 & *(Yellow) 5\\
    *(Blue) 3  & *(Violet) 7 & *(Violet) 8\\
    *(Yellow) 4 \\
    *(Violet) 6 \\
    \end{ytableau}
};

\node[anchor=south west, scale=0.7] at (0,0) {
    \begin{ytableau}
    *(Red) 1 & *(Red) 2 & *(Yellow) 5 & *(Violet) 8\\
    *(Blue) 3  & *(Violet) 7\\
    *(Yellow) 4 \\
    *(Violet) 6 \\
    \end{ytableau}

    \hspace{3mm}
    \begin{ytableau}
    *(Red) 1 & *(Red) 2 & *(Yellow) 5 & *(Violet) 7 & *(Violet) 8\\
    *(Blue) 3  \\
    *(Yellow) 4 \\
    *(Violet) 6 \\
    \end{ytableau}

    \hspace{3mm}
    \begin{ytableau}
    *(Red) 1 & *(Red) 2 & *(Yellow) 5 & *(Violet) 8\\
    *(Blue) 3  & *(Violet) 6 & *(Violet) 7 \\
    *(Yellow) 4 \\
    \end{ytableau}

    \hspace{3mm}
    \begin{ytableau}
    *(Red) 1 & *(Red) 2 & *(Yellow) 5 & *(Violet) 7 & *(Violet) 8\\
    *(Blue) 3 & *(Violet) 6 \\
    *(Yellow) 4 \\
    \end{ytableau}
};

\node at (6.8,-1) {Tableaux};
\draw[thick] (0,-0.2) -- (0,-0.4) -- (13.6,-0.4) -- (13.6,-0.2);
\end{scope}

\end{tikzpicture}
    \caption{The snake matroid \(\S(2,1,2,3)\) and the standard Young tableaux with its associated descent set.}
    \label{fig:snake tableaux}
\end{figure}

In what follows, we will prove that our standard Young tableaux satisfy the same recursion as Corollary \ref{cor:main identity snakes}.
\begin{proposition} \label{prop:syt recursion}
  Let \(\mathbf b = (b_1,\ldots,b_k,b_{k+1})\) be a composition and consider \(\mathbf b' = (b_1,\ldots, b_{k})\) and \(\mathbf b'' = (b_1,\ldots, b_{k-1}, b_k + b_{k+1})\). Then
  \begin{align*}
    s_{[b_{k+1}]} \cdot \sum_{\eta} \left|\SYT_{\eta}(\mathbf b')\right| s_{\eta} \ &= \ \sum_{\mu} \left|\SYT_{\mu}(\mathbf b)\right| s_{\mu}  + \sum_{\mu}\left|\SYT_{\mu}(\mathbf b'')\right| s_{\mu} \, ,
  \end{align*}
  where the sums are over all partitions of \(b_1 + \dots + b_k\).
\end{proposition}

\begin{proof}
  By Pieri's formula, the multiplication by \(s_{[b_{k+1}]}\) on the left-hand side corresponds to the mapping
  \begin{equation*}
    s_{\lambda} \ \longmapsto \ \sum_{\mu} s_{\mu} \, ,
  \end{equation*}
  where \(\mu\) runs over all partitions obtained by adding \(b_{k+1}\) boxes to \(\lambda\), at most one per column. Let \(\mathcal T\) denote the set of tableaux obtained by adding \(b_{k+1}\) empty boxes, at most one per column, to any tableau in \(\bigcup_{\lambda} \SYT_{\lambda}(\mathbf b')\), and let \(\mathcal S\) denote the set of shapes of \(\mathcal T\). Our goal is to find a bijection
  \begin{equation*}
    \mathcal T \ \longrightarrow \ \bigcup_{\mu \in \mathcal S} \SYT_{\mu}(\mathbf b) \cup \SYT_{\mu}(\mathbf b'') \, ,
  \end{equation*}
  and this is simple: just fill in the numbers \(\sum_{i=1}^k b_i + 1, \ldots, \sum_{i=1}^k b_i + b_{k+1}\) from left to right in the empty boxes of a tableau in \(\mathcal T\). If \(\sum_{i=1}^k b_i\) is a descent, the new tableau is contained in \(\SYT_{\mu}(\mathbf b)\) (for the appropriate shape \(\mu\)). Otherwise, the tableau is contained in \(\SYT_{\mu}(\mathbf b'')\). The mapping has an inverse: remove the entries \(\sum_{i=1}^k b_i + 1, \ldots, \sum_{i=1}^k b_i + b_{k+1}\). Since those entries are all ascents in the tableau we started with, the resulting tableau has at most one empty box per column, and obviously these boxes lie at the bottom of their columns. Hence the resulting tableau is contained in \(\mathcal T\). 
  
  Lastly, observe that for any
  \begin{equation*}
    T \in \bigcup_{\mu} \SYT_{\mu}(\mathbf b) \cup \SYT_{\mu}(\mathbf b'') \, ,
  \end{equation*}
  the tableau \(T|_{\sum_{i=1}^k b_i}\) is contained in \(\SYT_{\lambda}(\mathbf b')\) (for the appropriate \(\lambda\)). Hence, the bijection above proves the desired result.
\end{proof}

\begin{proof}[Proof of Theorem \ref{thm:combinatorial Sc snakes}]
    We prove the claim by induction on \(k\), the rank of the snake matroid. For \(k = 1\), we have 
    \begin{equation*}
        \Sc^c(\S(b)) \ = \ s_{[b]} \ = \ \sum_{\eta} |\SYT_{\eta}(b)| s_{\eta} \, .
    \end{equation*}
    Now, assuming by induction that the statement holds for \(k\geq 1\),
    \begin{align*}
        \Sc^c(\S(\mathbf b)) \ &= \ s_{[b_{k+1}]} \cdot \Sc^c(\S(\mathbf b')) - \Sc^c(\S(\mathbf b'')) \\
        &= \  s_{[b_{k+1}]} \cdot \sum_{\eta} |\SYT_{\eta}(\mathbf b')| s_{\eta}  - \sum_{\eta} |\SYT_{\eta}(\mathbf b'')| s_{\eta} \\
        &= \ \sum_{\eta} |\SYT_{\eta}(\mathbf b)| s_{\eta} \, ,
    \end{align*}
    where the first equality is Corollary \ref{cor:main identity snakes}, the second follows by induction, and the third from Proposition~\ref{prop:syt recursion}.
\end{proof}
\begin{example}
    \label{ex:recover beta snake}
    It is well known that the \(\beta\)-invariant of a snake matroid is one, we recover this fact. Let \(\S = \S(\mathbf b)\) be a snake matroid of rank \(k\) on \([n]\). Recall that the beta invariant is given by \(\beta(\S) = d_{h^c}(\S)\) where \(h = [n-k, 1^{k-1}]\). The only standard Young tableau in \(\SYT_h(\mathbf b)\) is the one containing \(1,b_1+2,b_2+1,\dots, b_{k-1}+1\) in the first column, and the remaining integers in \([n-1]\) in the first row. Hence by Theorem \ref{thm:combinatorial Sc snakes} we have
    \[
        \beta(\S) \ = \ d_{h^c}(\S) \ = \ |\SYT_h(\mathbf b)| \ = \ 1 \, .
    \]
\end{example} 

\begin{example}
    \label{ex:recover minimal}
    The minimal matroid \(\mathsf{T}_{k,n}\) is the unique, up to isomorphism, connected matroid of rank \(k\) on \([n]\) with minimal number of bases. It is isomorphic to either of the snake matroids \(\S(n-k,1,1,\dots, 1)\) and \(\S(1,1,\dots, 1,n-k)\) with \(k-1\) ones. We recover \cite[Lemma 5.3]{JP} which states that \(\Sc(\mathsf{T}_{k,n}) = s_{h^c}\). By Theorem \ref{thm:combinatorial Sc snakes} this is easy to see since for \(\S(n-k,1,1,\dots,1)\) the only standard Young tableau with descent set \(\{n-k, n-k+1, \dots , n-2\}\) is the one of shape \(h\) with the entries \(1,2,\dots, n-k\) in the first row and \(1,n-k+1,\dots, n-1\) in the first column.
\end{example}

\begin{remark} \label{rmk:Vasu}
    In \cite[Theorem 3]{gessel} Gessel shows that when expanding a symmetric polynomial \(g\) in the basis of \emph{fundamental quasisymmetric polynomials}, the coefficients are given by pairing \(g\) with ribbon Schur polynomials.
    A consequence is \cite[Theorem~7]{gessel} which with our notation states that \(s_{\rho(\mathbf{b})} = \sum_{\eta} |\SYT_\eta(\mathbf{b})|s_\eta\). As mentioned previously, together with Theorem~\ref{thm:ribbon snake} this provides a proof of Theorem~\ref{thm:combinatorial Sc snakes}. In \cite[Section~10]{NadeauSpinkTewari} the authors provide a geometric interpretation of \cite[Theorem~3]{gessel} in terms of \emph{Grassmannian Richardson varieties} \(Y^\lambda_\mu\) where \(\lambda/\mu\) is a ribbon. The precise connection between \(Y^\lambda_\mu\) and the torus orbit closure \(\overline{Tx}\) when \(\M_x\) is the snake matroid corresponding to the skew shape \(\lambda/\mu\) is worthy of further investigation.
\end{remark}

\subsection{Schubert coefficients of lattice path matroids}\label{sec:lpm}
We now extend Theorem \ref{thm:combinatorial Sc snakes} to give a combinatorial formula for the Schubert coefficients of connected lattice path matroids, in particular nested matroids and uniform matroids.
\begin{remark}
    While a connected lattice path matroid can also be described by a skew shape \(\lambda/\mu\), it is in general not true that \(\Sc^c(\M(\lambda/\mu)) = s_{\lambda/\mu}\). Indeed the equality holds if and only if the matroid is a snake, as any other shape would have too many boxes for the equality to hold.
\end{remark}

Let \(\M = \M(\lambda / \mu)\) be a connected lattice path matroid of rank \(k\) on \([n]\). Let \(\S_\mathbf{U}\) and \(\S_\mathbf{L}\) be, respectively, the uppermost and lowermost snake matroid of rank \(k\) on \([n]\) whose diagram fit inside \(\lambda / \mu\). Let \(D_\mathbf{U} = \{c_1,\ldots,c_{k-1} \}\) and \(D_\mathbf{L} = \{d_1,\ldots,d_{k-1}\}\) be the descent sets associated to \(\S_\mathbf{U}\) and \(\S_\mathbf{L}\). Clearly \(c_i \leq d_i\), so we may consider the set system of intervals 
\[
    \mathcal A \ = \ \SetOf{[c_i,d_i]}{i = 1,2,\dots, k-1} \, .
\]
%Figure \ref{fig:LPM setsystem} shows a lattice path matroid and the corresponding set system.
%\begin{figure}
%    \centering
%    \scalebox{0.6}{\input{figures/LPMtransversal}}
%    \caption{The lattice path matroid \(M(\lambda /\mu)\) with \(\lambda = [6^2,5,3]\) and \(\mu = [3,1^2]\). The sets in the corresponding set system \(\mathcal A\) are the labels in each blue strip.}
%    \label{fig:LPM setsystem}
%uni\end{figure}
\begin{theorem}
    \label{thm:combinatorial Sc lattice path}
    Let \(\M\) be a connected lattice path matroid of rank \(k\) on \([n]\) and let \(\eta\) be a partition of \(n-1\) in \(k\times (n-k)\). Then with notation as above, the Schubert coefficient of \(\M\) associated to \(\eta^c\) is given by 
    \[
        d_{\eta^c}(\M) \ = \left|\SetOf{T \in \SYT_\eta}{\Des(T) \text{ is a transversal of } \mathcal{A}}\right| \, .
    \]
\end{theorem}

\begin{proof}
    By Proposition~\ref{prop:valuative into snakes} and the fact that \(d_{\eta^c}\) is an additive, valuative matroid invariant we have that 
    \[
        d_{\eta^c}(\M) \ = \ \sum_{\S}d_{\eta^c}(\S) \, ,
    \]
    where the sum is over all snake matroids of rank \(k\) on \([n]\) whose diagram fits in the diagram of \(\M\). The map sending such a snake matroid to its associated descent set is a bijection with the set of transverals of \(\mathcal A\). Now Theorem \ref{thm:combinatorial Sc snakes} finishes the proof.
\end{proof}
We now describe explicitly the Schubert coefficients of nested matroids. Let \(\SYT_\eta^{k-1}\) denote the set of standard Young tableaux of shape \(\eta\) with \(k-1\) descents.
\begin{corollary}
    \label{cor:combinatorial Sc nested}
    Let \(\N\) be a connected nested matroid of rank \(k\) on \([n]\) with cyclic flats \(\emptyset = H_0 \subset H_1 \subset \dots \subset H_s = E\) and let \(r_i = \rank(H_i)\)  and \(h_i = |H_i|\). Let \(\eta\) be a partition of \(n-1\) in \(k\times (n-k)\). Then the Schubert coefficient of \(\N\) associated to \(\eta^c\) is given by
    \[
        d_{\eta^c}(\N) \ = \ \left|\SetOf{T \in \SYT_\eta^{k-1}}{ \des(T|_{h_i}) < r_i \ \text{ for } \ i = 1,\dots, s-1}\right| \, .
    \]
\end{corollary}
\begin{proof}
    Up to isomorphism we may assume that \(H_i = \{1,2,\dots, h_i\}\). Then \(\N\) is the lattice path matroid \(\M[\mathbf{L}, \mathbf{U}]\) with lower path \(\mathbf{L} = \mathbf{E}^{n-k}\mathbf{N}^k\) and upper path
    \[
        \mathbf{U} \ = \ \mathbf{N}^{r_1}\mathbf{E}^{h_1-r_1}\mathbf{N}^{r_2-r_1} \mathbf{E}^{h_2 - h_1 - (r_2-r_1)} \dots \ \mathbf{N}^{r_s - r_{s-1}} \mathbf{E}^{h_s - h_{s-1} - (r_s - r_{s-1})} \, .
    \]
    With notation as in Theorem \ref{thm:combinatorial Sc lattice path} the descent set of the uppermost snake matroid is the disjoint union of intervals
    \[
        D_\textbf{U}  \ = \ [1, r_1-1] \cup [h_1, h_1+r_2-r_1-1] \cup \dots \cup [h_{s-1}, h_{s-1} + r_s-r_{s-1}-1] \, ,
    \]
    and the descent set of the lowermost snake is the interval \(D_\mathbf{L} = [n-k,n-2]\). Let \(\mathcal A\) be the associated set system. By Theorem \ref{thm:combinatorial Sc lattice path} it only remains to show that 
    \[
        \SetOf{T \in SYT_\eta}{\Des(T) \text{ is a transversal of } \mathcal{A}} \ = \ \SetOf{T \in \SYT_\eta^{k-1}}{ \des(T|_{h_i}) < r_i} \, .
    \]
    Let \(T\) be a standard Young tableaux in the set on the left above. Then for each \(i\) the descent set of the restriction \(T|_{h_i}\)  is 
    \[
        \Des(T|_{h_i}) \ = \ \Des(T) \cap [1,h_i-1] \, .
    \]
    Only the first \(r_i-1\) intervals in \(\mathcal A\) has nonempty intersection with \([1,h_i-1]\) so \(\des(T|_{h_i}) < r_i\). For the other direction let \(T \in \SYT_\eta^{k-1}\) be such that \(\Des(T)\) is not a transversal of \(\mathcal A\). Then \(\Des(T)\) is the descent set associated to some snake matroid \(\S\) of rank \(k\) on \([n]\) whose diagram passes above the diagram of \(\N\). Let \(i\) be the smallest index such that part of the diagram of \(\S\) is to the north--west of the point \((h_i-r_i,r_i)\). Then \(T|_{h_i}\) has at least \(r_i\) descents. Figure \ref{fig:nested matroid proof} depicts a nested matroid with the intervals \([c_i,d_i]\) in \(\mathcal A\) and the points \((h_i-r_i,r_i)\) marked.
\end{proof}
\begin{figure}
    \centering
    \scalebox{0.6}{\begin{tikzpicture}
    \foreach \x\y in {2/0,2/1,4/2,4/3,5/4}
        \fill[blue!30, rounded corners=0.5mm] (\x+0.2,\y+0.2) rectangle (6.8,\y+0.8);

    \foreach \x\y in {0/0,2/1,2/2,4/3,4/4,5/5}
        \draw[thick] (\x,\y) -- (7,\y);
        
    \foreach \x\y in {1/1,2/1,3/3,4/3,5/5,6/6,7/6}
        \draw[thick] (\x,0) -- (\x,\y);

    \draw[thick] (0,0) -- (0,1) -- (2,1) -- (2,3) -- (4,3) -- (4,5) -- (5,5) -- (5,6) -- (7,6);

    \draw[line width=0.6mm, rounded corners=0.5mm, Red] (2,1) -- (3,1) -- (3,-0.01415);
    \draw[line width=0.6mm, rounded corners=0.5mm, Red] (4,3) -- (5,3) -- (5,2) -- (6,2)  -- (6,1)  -- (7,1) -- (7,-0.01415);
    \draw[line width=0.6mm, rounded corners=0.5mm, Red] (5,5) -- (6,5) -- (6,4) -- (7.01415,4);

    \node[very thick, dot, Red] at (0,0){};
    \node[Red] at (-0.5,0.3){\(0,0\)};
    
    \node[thick, dot, Red] at (2,1){};
    \node[Red] at (1.5,1.3){\(3,1\)};

    \node[very thick, dot, Red] at (4,3){};
    \node[Red] at (3.5,3.3){\(7,3\)};
    
    \node[very thick, dot, Red] at (5,5){};
    \node[Red] at (4.4,5.3){\(10,5\)};
    
    \node[very thick, dot, Red] at (7,6){};
    \node[Red] at (6.4,6.3){\(13,6\)};

    \foreach \x in {1,...,7}
        \node at (\x-0.5,0.5) {\(\x\)};

    \foreach \x in {4,...,8}
        \node at (\x-1.5,1.5) {\(\x\)};

    \foreach \x in {5,...,9}
        \node at (\x-2.5,2.5) {\(\x\)};

    \foreach \x in {8,9,10}
        \node at (\x-3.5,3.5) {\(\x\)};

    \foreach \x in {9,10,11}
        \node at (\x-4.5,4.5) {\(\x\)};

    \foreach \x in {11,12}
        \node at (\x-5.5,5.5) {\(\x\)};
\end{tikzpicture}}
    \caption{A lattice path representation of a nested matroid. The inner corners of the upper path are the coordinates \((h_i-r_i,r_i)\), here depicted in red with the corresponding \(h_i\) and \(r_i\). The intervals \([c_i,d_i]\) are marked with blue rectangles in each row. The possible restrictions \(T|_{h_i}\) are obtained by building standard Young tableaux as in Theorem \ref{thm:combinatorial Sc snakes} up to the corresponding red line.}
    \label{fig:nested matroid proof}
\end{figure}
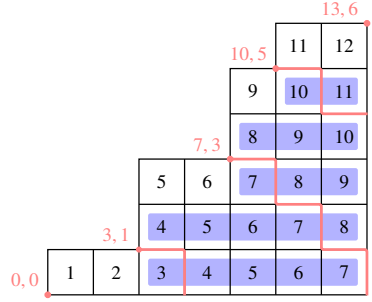

\begin{remark}
    Note that if a partition \(\eta\) of \(n-1\) is not contained in \(k\times (n-k)\) then any standard Young tableaux \(T\) of shape \(\eta\) cannot have exactly \(k-1\) descents. If \(\eta\) has more than \(k\) parts then \(T\) has at least \(k\) descents, and if \(\eta_1 > n-k\) then \(T\) has at most \(k-2\) descents. All the standard Young tableaux involved in Theorems \ref{thm:combinatorial Sc snakes}, \ref{thm:combinatorial Sc lattice path} and Corollaries \ref{cor:combinatorial Sc nested} and \ref{cor:combinatorial Sc uniform}  have exactly \(k-1\) descents. So the partitions of \(n-1\) that may give a nonzero contribution are those contained in \(k\times (n-k)\).
\end{remark}

We finally focus on uniform matroids. Klyachko provided formulas for the Schubert coefficients of uniform matroids in \cite{Klyachko} and \cite{klyachko-r}\footnote{The formula in \cite{Klyachko} is an expression with alternating signs while the formula in \cite{klyachko-r} (see also Corollary \ref{cor:combinatorial Sc uniform}) is clearly positive, however the article is less known as it is published in Russian.}. The formula in the latter was reproved independently by Nadeau and Tewari in \cite{nadeau-tewari}.

\begin{corollary}[{\cite[Theorem 3.3]{klyachko-r}}]
    \label{cor:combinatorial Sc uniform}
    Let \(\eta\) be a partition of \(n-1\) in \(k\times (n-k)\). The Schubert coefficient of the uniform matroid~\(\U_{k,n}\) associated to \(\eta^c\) is given by
    \[
        d_{\eta^c}(\U_{k,n}) \ = \ |\SYT_\eta^{k-1}| \, .
    \]
\end{corollary}

\begin{proof}
    The result follows from \ref{cor:combinatorial Sc nested} since the uniform matroid \(\U_{k,n}\) is the connected nested matroid with cyclic flats \(\emptyset = H_0 \subset H_1 = [n]\).
\end{proof}

\begin{remark}
     Klyachko's motivation for Corollary \ref{cor:combinatorial Sc uniform} was to compute the Chow class of the torus orbit closure of a generic point in the Grassmannian. Our Corollary \ref{cor:combinatorial Sc nested} improves his results as the matroids associated to generic points in a Schubert variety are precisely nested matroids.
\end{remark}

In \cite[Corollary~5.4]{berget-fink-equivariant}, Berget and Fink provide the following alternative description of $\Sc(\U_{k,n})$.
\begin{proposition}\label{prop:berget-fink-uniform}
    For any $k\leq n$, the Chow class of the uniform matroid $\U_{k,n}$ equals
    \[
    \Sc(\U_{k,n}) = \sum_{\mu}s_{\mu}s_{\widetilde{\mu}},
    \]
    where the sum is taken over all partitions \(\mu\subseteq (k-1)\times (n-k-1)\) and \(\widetilde{\mu}\) denotes the complement of \(\mu\) inside the \((k-1)\times(n-k-1)\) rectangle. 
\end{proposition}

We know that by fixing the rectangle \(k\times (n-k)\), a partition \(\mu \subseteq (k-1)\times(n-k-1)\) uniquely identifies a partition \(\lambda = [n-k, \mu_1 + 1,\ldots, \mu_{k-1}+1] \subseteq k \times (n-k)\) such that \(\lambda/\mu\) is a ribbon. It is straightforward to check that \(\lambda^c = \widetilde{\mu}\). The decomposition in Proposition \ref{prop:berget-fink-uniform} agrees with the following result.
\begin{proposition}\label{prop:Sc snake as a product}
    Let \(k \leq n\) and let \(\mu\) and \(\lambda\) be as above, then
    \[
        \Sc(\S(\lambda/\mu)) \ = \ s_{\mu}s_{\lambda^c} \, .
    \]
\end{proposition}
\begin{proof}
    By Theorem \ref{thm:ribbon snake} and after expanding in the Schur basis, the result is equivalent to showing that
    \[
    \sum_{\eta} c^\lambda_{\mu,\eta^c}s_{\eta} \ = \ \sum_{\eta}c^\eta_{\lambda^c,\mu}s_\eta \, .
    \]
    We then need to show that for a fixed partition \(\eta \subseteq k\times (n-k)\), the equality \(c^\lambda_{\mu,\eta^c} = c^\eta_{\lambda^c,\mu}\) holds. By the symmetry of the Littlewood--Richardson coefficients, this is equivalent to \(c^\lambda_{\eta^c,\mu} = c^\eta_{\lambda^c,\mu}\), which follows after observing that \(s_{\lambda/\eta^c} = s_{\eta/\lambda^c}\). To see that, notice that \(\eta/\lambda^c\) is obtained by rotating the diagram of \(\lambda/\eta^c\) by 180 degrees.  
\end{proof}

\begin{remark} \label{rmk:Carl Lian}
    In light of Proposition \ref{prop:Sc snake as a product}, Proposition \ref{prop:berget-fink-uniform} reduces to a special case of Proposition \ref{prop:valuative into snakes}. Motivated by the two formulas by Klyachko and Berget and Fink, in \cite{lian}, Lian studied the expansion of \(s_{\mu}s_{\widetilde{\mu}}\) in the Schur basis and described the coefficients that appear using a combinatorial count of so-called \emph{1-strip-less-tableaux}. The coefficients appearing in \cite[Theorem 1.3]{lian} agree, after running the standardization algorithm described in \cite[Appendix~A]{lian}, with the ones in Theorem~\ref{thm:combinatorial Sc lattice path}. Proposition~\ref{prop:Sc snake as a product} gives a new interpretation of the summands in Proposition~\ref{prop:berget-fink-uniform} as Chow classes of snake matroids.
\end{remark}

\section{Applications}\label{sec: applications}

\subsection{Robinson–Schensted–Knuth correspondence (RSK) and volume} \label{subsec:rsk and volume}
The volume of snake matroid polytopes was studied in great details in \cite{Knauer-snakes}. There, the authors show the following result. Consider the poset whose Hasse diagram is obtained by rotating the ribbon \(\ribbon(\mathbf b)\) defining a snake matroid~\(\S\) by \(45\) degrees clockwise, with an element for every box in \(\ribbon(\mathbf b)\) and a cover relation whenever two boxes are adjacent. Let \(Z(\S)\) be the dual of this poset and label it from left to right as depicted in Figure \ref{fig:snake poset}.
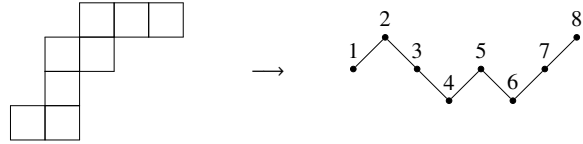
\begin{figure}[ht]
    \centering
    \scalebox{0.7}{ %\begin{tikzpicture}
 %   \node[anchor=south west] at (0,0) {
 %     \ydiagram{4+1, 1+4, 1+1, 2}
 %   };
 %   \node at (5,1.45) {\(\longrightarrow\)};

 %   \coordinate (shift) at (6,1);
 %   \begin{scope}[shift=(shift), scale=0.6]
 %       \foreach \x\y in {1/0,2/1,3/0,4/-1,5/0,6/1,7/2,8/1}
 %           \node[dot, label={{\x}}] (\x) at (\x,\y) {};
 %       \draw[rpath] (1) -- (2) -- (3) -- (4) -- (5) -- (6) -- (7) -- (8);
 %   \end{scope}
%\end{tikzpicture}

 \begin{tikzpicture}
    \node[anchor=south west] at (0,0) {
      \ydiagram{2+3, 1+2, 1+1, 2}
    };
    \node at (5,1.45) {\(\longrightarrow\)};

    \coordinate (shift) at (6,1.5);
    \begin{scope}[shift=(shift), scale=0.6]
        \foreach \x\y in {1/0,2/1,3/0,4/-1,5/0,6/-1,7/0,8/1}
            \node[dot, label={{\x}}] (\x) at (\x,\y) {};
        \draw[rpath] (1) -- (2) -- (3) -- (4) -- (5) -- (6) -- (7) -- (8);
    \end{scope}
\end{tikzpicture}}
    \caption{The snake matroid \(\S = \S(2,1,2,3)\) and the corresponding labeled poset \(Z(\S)\).}
    \label{fig:snake poset}
\end{figure}

\begin{theorem}[{\cite[Theorem~4.7]{Knauer-snakes}}]\label{thm:volume snake linear extensions}
    The normalized volume of a snake matroid \(\S(\mathbf{b})\) equals the number of linear extensions of the poset \(Z(\S)\). In particular, this is equal to the number of permutations \(\pi\) in \(\mathfrak{S}_{n-1}\) such that their descent set is equal to \(\Des(\pi) = \Des(\mathbf{b})\).
\end{theorem}

For the second claim in the statement, recall that the linear extensions of \(Z(\S)\) are given by the Jordan--Hölder set
\begin{equation*}
    \mathcal{JH}(Z(\S)) \ \coloneqq \ \SetOf{\pi \in \mathfrak S_{n-1}}{a \prec_{Z(\S)} b \implies \pi^{-1}(a) < \pi^{-1}(b)} \, .
\end{equation*}
Hence, the inverses of the linear extensions of \(Z(\S)\) are the permutations \(\pi\) such that if \(a \prec_{Z(\S)} b\), then \(\pi(a) < \pi(b)\). Given our chosen labelling of \(Z(\S)\), this is exactly the set of permutations having \(\Des(\mathbf b)\) as descent set. 
We show now how to recover this result in our setting. We will use the following linear relation that ties the Schubert coefficients of a matroid \(\M\) to the normalized relative volume of its matroid polytope \(\Volume(\M)\).
\begin{lemma}[{\cite[Proposition 5.1]{JP}}]
    \label{lem:important}
    Let \(\M\) be a matroid of rank \(k\) on \([n]\). Then
    \[
        \sum_\eta |\SYT_{\eta^c}| \ d_\eta(\M) \ = \ \Volume(\M) \, .
    \]
\end{lemma}

\begin{proof}[Proof of Theorem \ref{thm:volume snake linear extensions}]
    Given a snake matroid \(\S = \S(\mathbf b)\) we write
    \begin{align*}
    \Volume(\S(\mathbf b)) \ &= \ \sum_{\eta \vdash n-1} |\SYT_{\eta}| \cdot d_{\eta^c}(\S) \\
    &= \ \sum_{\eta \vdash n-1} |\SYT_{\eta}| \cdot |\SYT_{\eta}(\mathbf b)| \\
    &= \ |\SetOf{\pi \in \mathfrak{S}_{n-1}}{\Des(\pi) = \Des(\mathbf b)} |\, ,
\end{align*}
where the first equality is Lemma \ref{lem:important}, the second is Theorem \ref{thm:combinatorial Sc snakes}, and the third follows from the RSK correspondence, see \cite[Section 7.11]{Stanley-EC2} for more details. 
\end{proof}

In \cite[Theorem 5.4]{benedetti-knauer-valenciaporras} the authors give a recursive formula for the volume of lattice path matroids of rank 2. The formula for the volume of snake matroids from Theorem \ref{thm:volume snake linear extensions} can be extended to all lattice path matroids, see also \cite[Corollary~5.1.1]{benedetti-knauer-valenciaporras}.
\begin{corollary}
    Let \(\M\) be a connected lattice path matroid, and let \(D_\mathbf{L} = \{c_1, \ldots, c_{k-1}\}\) and \(D_\mathbf{U} = \{d_1,\ldots, d_{k-1}\}\) be the descent sets associated to the lowermost and uppermost snakes fitting inside \(\M\), respectively. Then \(\Volume(\M)\) equals the number of permutations \(\pi \in \mathfrak S_{n-1}\) such that \(\Des(\pi)\) is a transversal of the set system
    \begin{equation*}
        \mathcal A \ = \ \SetOf{[c_i,d_i]}{i = 1,2,\dots, k-1} \, .
    \end{equation*}
\end{corollary}

\begin{proof}
    As above, this follows from Theorem \ref{thm:combinatorial Sc lattice path}, Lemma \ref{lem:important} and RSK.
\end{proof}

In \cite{Gessel-Viennot}, Gessel and Viennot provide a determinantal formula to compute the number of permutations with a given descent set. Let \(A = ( a_{ij})_{i,j\geq 0}\) be the infinite matrix where \( a_{ij} = \binom{i}{j}\) for \(j \leq i\) and \(0\) for \(j > i\). A \emph{binomial determinant} is any minor of \(A\). The minor corresponding to rows \(0 \leq a_1 < a_2 < \ldots < a_k\) and columns \(0 \leq b_1 < b_2 < \ldots < b_k\) is denoted by
\[
\binom{a_1,\ldots, a_k}{b_1, \ldots, b_k} \ = \ \det\left[\binom{a_i}{b_j} \right]_{1\leq i,j \leq k}\, .
\]
\begin{theorem}[{\cite[Corollary 6]{Gessel-Viennot}}]\label{thm: Gessel-Viennot}
    If \(\mathcal D = \{d_1,\ldots, d_k\} \subseteq [n-2]\), the number of permutations in \(\mathfrak{S}_{n-1}\) whose descent set is \(\mathcal D\) is 
    \[
        \binom{\mathcal D \cup \{n-1\}}{\{0\}\cup \mathcal D} \, .
    \]
\end{theorem}

We prove Theorem \ref{thm: Gessel-Viennot} using our results.
\begin{proof}[Proof of Theorem \ref{thm: Gessel-Viennot}]
    Let \(\mathbf{b} = (b_1,\dots, b_k)\) be the composition such that the snake matroid \(\S = \S(\mathbf{b})\) of rank \(k\) on \([n]\) has associated descent set \(\Des(\mathbf{b}) = \mathcal D\). First using Theorem \ref{thm:volume snake linear extensions} and Lemma \ref{lem:important} we have
    \begin{align*}
        |\{\pi \in \mathfrak S_{n-1} \mid \Des(\pi) = \mathcal D\}| \ &= \ \Volume(\S) \ = \ \sum_{\eta} d_{\eta^c}(\S) |\SYT_\eta|\\
        &= \ \sum_{\eta} d_{\eta^c}(\S) [u^{[n-1]}]s_\eta\\
        &= \ [u^{[n-1]}]\left(\sum_{\eta} d_{\eta^c}(\S) \ s_\eta\right) \, .
    \end{align*}
    In the third equality we use that the coefficient of \(u^{[n-1]} = u_1\cdots u_{n-1}\) of the Schur function \(s_\eta\) is the number of standard Young tableaux of shape \(\eta\).
    We will show by induction on the rank \(k\) that for any snake matroid \(\S\) on \([n]\) with associated descent set \(\mathcal D\) we have
    \begin{equation}
        \label{eq:GV induction}
        [u^{[n-1]}]\left(\sum_{\eta} d_{\eta^c}(\S) \ s_\eta\right) \ = \ \binom{\mathcal{D} \cup \{n-1\}}{\{0\}\cup \mathcal D} \, .
    \end{equation}
    For \(k = 1\), that is \(\S = \S(n-1)\) and \(\mathcal D = \emptyset\), we have that \eqref{eq:GV induction} simplifies to 
    \[
        [u^{[n-1]}]s_{[n-1]} \ = \ 1 \ = \ \binom{n-1}{0} \, .
    \]
    For \(k \geq 2\) let \(\mathbf{b} = (b_1,\dots, b_k)\) be a composition of \(n-1\) and let \(\mathcal D = \Des(\mathbf b)\) be the associated descent set. Consider the compositions \(\mathbf{b}' = (b_1,\dots b_{k-1})\) and \(\mathbf{b}'' = (b_1,\dots,b_{k-1} + b_k)\) of \(n-b_k-1\) and \(n-1\) respectively. Notice that the associated descent set to both is \(\mathcal D' = \{d_1,\dots,d_{k-2}\}\). Now using Corollary \ref{cor:main identity snakes} we get
    \begin{align*}
        & [u^{[n-1]}]\left(\sum_{\eta} d_{\eta^c}(\S(\mathbf{b})) \ s_\eta\right) \ = \ [u^{[n-1]}] \left( s_{[b_k]} \sum_{\mu}d_{\mu^c}(\S(\mathbf{b}'))s_\mu - \sum_{\eta} d_{\eta^c}(\S(\mathbf{b}''))s_{\eta}\right)\\
        & = \ \sum_{A\in \binom{n-1}{b_k}} [u^A]s_{[b_k]} \ [u^{[n-1]\setminus A}]\left( \sum_{\mu}d_{\mu^c}(\S(\mathbf{b}'))s_\mu \right) \ - \ [u^{[n-1]}]\left(\sum_{\eta} d_{\eta^c}(\S(\mathbf{b}''))s_{\eta}\right)\\
        & = \ \binom{n-1}{b_k} [u^{[n-b_k-1]}]\left( \sum_{\mu}d_{\mu^c}(\S(\mathbf{b}'))s_\mu \right) \ - \ [u^{[n-1]}]\left(\sum_{\eta} d_{\eta^c}(\S(\mathbf{b}''))s_{\eta}\right)\\
        & = \ \binom{n-1}{b_k} \binom{\mathcal D' \cup \{n-b_k-1\}}{\{0\} \cup \mathcal{D}'} - \binom{\mathcal D' \cup \{n-1\}}{\{0\} \cup \mathcal{D}'} \ = \ \binom{\mathcal D \cup \{n-1\}}{\{0\} \cup \mathcal{D}} \, .
    \end{align*}
    In the third equality we used that \([u^A]s_{[b_k]} = 1\) and that \([u^{[n-1]\setminus A}]s_\mu\) is independent of~\(A\). The fourth equality is the induction hypothesis and the last equality follows from expanding the minor \(\binom{\mathcal D \cup \{n-1\}}{\{0\} \cup \mathcal{D}}\) along the last column.
\end{proof}

\subsection{Support}\label{sec:support}
A starting point to tackle Conjecture \ref{conj:nonnegativity} would be to know for which partitions the corresponding Schubert coefficients of a given matroid are nonzero. We define the support of a matroid \(\M\) of rank \(k\) on \([n]\) to be the set
\[
    \supp(\M) \ = \  \left\{\eta \vdash n - \kappa(\M) \mid \ d_{\eta^c}(\M) \neq 0\right\} \, .
\] 
When restricting to snake matroids, the support of ribbon Schur functions has been the subject of much work, see for example \cite{McNamara,king-welsh-vanwilligenburg,McNamara-vanWilligenburg} and references therein. Finding a complete characterization of the support is considered a hard problem \cite{McNamara-Willigenburg}.

In {\cite[Proposition~3.1]{McNamara}}, McNamara provides a necessary condition for a partition \(\eta\) to be in the support of a ribbon Schur function \(s_{\ribbon(\mathbf b)}\). We record the result here in our language. For a ribbon \(\ribbon(\mathbf b)\), let \(\operatorname{rows}(\mathbf b)\), respectively \(\operatorname{cols}(\mathbf b)\), denote the partition obtained by ordering the lengths of the rows, respectively columns, of \(\ribbon(\mathbf b)\) in weakly decreasing order.
\begin{lemma} \label{lem:mcnamara support}
    Let \(\S = \S(\mathbf b)\) be a snake matroid of rank \(k\) on \([n]\), and let \(\eta \subseteq k \times (n-k)\) be a partition of \(n-1\). If \(\eta \in \supp(\S)\), then
    \[
        \operatorname{rows}(\mathbf b) \leq \eta \leq \operatorname{cols}(\mathbf b)^t
    \]
    in the dominance order. 
\end{lemma}

When \(\eta\) has full first row or full first column Corollary \ref{cor:Kostka} lets us give a complete characterization, as it is known that \(K_{\eta,\mathbf b}\) is nonzero if and only if \(\operatorname{rows}(\mathbf b) \leq \eta\) in the dominance order (\cite[Proposition 7.10.5]{Stanley-EC2}).

\begin{proposition}\label{prop: dominance order}
    Let \(\S = \S(\mathbf b)\) be a snake matroid of rank \(k\) on \([n]\) and \(\eta \vdash n-1\) be a partition of length \(k\). Then
    \[
        \eta \in \supp(\S) \iff \operatorname{rows}(\mathbf b) \leq \eta \text{ in dominance order.} 
    \]
    Similarly, if \(\eta_1 = n-k\),
    \[
        \eta \in \supp(\S) \iff \eta \leq \operatorname{cols}(\mathbf b)^t \text{ in dominance order.} 
    \]
\end{proposition}

The following result gives a slightly different necessary condition.

\begin{proposition}\label{prop:support}
    Let \(\S\) be a snake matroid given by the ribbon \(\lambda / \mu\), and \(\eta\) a partition. If \(\eta \in \supp(\S)\) then \(\eta \subseteq \lambda\).
\end{proposition}

\begin{proof}
    We will show that if \(\eta \not \subseteq \lambda\) then \(\eta \not \in \supp(\S)\) by induction on the rank \(k\) of \(\S\).
    The base case \(k=1\) is trivial as \(\Sc(\S(b)) = \Sc(\U_{1,b+1}) = \sigma_{[b]}\).
    when \(k > 1\), we consider the snake matroid \(\S' = \S(b_1,\ldots,b_{k-1})\) on \([n-b_k]\) given by the skew diagram \(\lambda' / \mu'\) where \(\lambda' = [\lambda_2, \dots, \lambda_k]\) and \(\mu' = [\mu_2, \dots ,\mu_k]\). For \(\eta \in \supp(\S)\), by Theorem \ref{thm:combinatorial Sc snakes}, \(d_{\eta^c}(\S) = |\SYT_\eta(\S)|\). These standard Young tableaux can be built from those in \(\SYT_{\eta'}(\S')\), for some \(\eta' \in \supp(\S')\), by adding \(b_k\) entries at most one per column, such that the smallest new element added \(n-b_k\) makes \(n-b_k-1\) a descent. By induction all these \(\eta'\) are contained in \(\lambda'\). It is sufficient to show that a tableau of shape \(\nu\) obtained like this from a tableau of shape \(\lambda'\) is contained in \(\lambda\). If \(\nu_1 > \lambda_1\), then we put all the new entries in the first row, which violates the descent condition. If \(\nu_i > \lambda_i\) for any other \(i\) this violates the condition of adding at most one entry per column, see Figure \ref{fig:snake support}. Hence \(\nu \subseteq \lambda\), and in particular \(\eta \subseteq \lambda\). 
\end{proof}

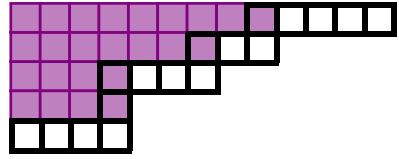
\begin{figure}[ht]
    \centering
    \begin{tikzpicture} 

    \node[violet, anchor=south west] at (0,12pt) {
        \ytableausetup{boxsize = 10pt, boxframe = 1pt}
        \begin{ytableau}
            *(Violet) & *(Violet) & *(Violet) &  *(Violet) &  *(Violet) &  *(Violet) & *(Violet) & *(Violet) & *(Violet)\\
            *(Violet) & *(Violet) & *(Violet) & *(Violet) & *(Violet) & *(Violet) & *(Violet)\\
            *(Violet) & *(Violet) & *(Violet) & *(Violet)\\
            *(Violet) & *(Violet) & *(Violet) & *(Violet)\\
        \end{ytableau}
        \ytableausetup{boxsize = normal, boxframe = normal}
    };
    
    \node[anchor=south west] at (0,0) {
        \ytableausetup{boxsize = 9pt, boxframe = 2pt}
        \begin{ytableau}
            \none & \none & \none  & \none  & \none & \none & \none & \none  & *(Violet) & *(white) & *(white) & *(white) & *(white)\\
            \none & \none & \none &  \none &  \none &  \none & *(Violet) & *(white) & *(white)\\
            \none & \none & \none & *(Violet) & *(white) & *(white) & *(white)\\
            \none & \none & \none & *(Violet)\\
            *(white) & *(white) & *(white) & *(white)\\
            \end{ytableau}
        \ytableausetup{boxsize = normal, boxframe = normal}
    };
\end{tikzpicture}
    \caption{The snake matroid \(\S(4,1,4,3,5)\) and \(\lambda' = [9,7,4,4]\) in violet.}
    \label{fig:snake support}
\end{figure}

As the following example shows, our condition is not implied by Lemma \ref{lem:mcnamara support}. The implication fails even for relatively small snake matroids.

\begin{example}
    Let \(\S = \S(\mathbf b)\) with \(\mathbf b = (1,2,3)\). We have that \(\operatorname{rows}(\mathbf b) = [3,2,1]\) and \(\operatorname{cols}(\mathbf b) = [2,2,1,1]\), so \(\operatorname{cols}(\mathbf b)^t = [4,2]\). Consider the partition \(\eta = [3,3]\). Clearly, 
    \[
        \operatorname{rows}(\mathbf b) \leq \eta \leq \operatorname{cols}(\mathbf b)^t
    \]
    in the dominance order, but \(\eta\) is not contained in \(\lambda\) and therefore the corresponding Schubert coefficient should be zero by Proposition \ref{prop:support}.
\end{example}

Recall that by inverting the order of the entries of \(\mathbf b\) we obtain an isomorphic snake matroid \(\S' = \S(\mathbf b')\), whose diagram can be obtained by rotating the diagram of \(\S\) by 180 degrees. Since \(\Sc\) is a matroid invariant, we know that \(\Sc(\S) = \Sc(\S')\), so this reversing operation can sometimes lead to even better bounds on the support.

\begin{example}
    Consider the snake matroid \(\S  = \S(2,1,5,1)\). By Theorem \ref{thm:combinatorial Sc snakes} we may compute the Schubert coefficients of \(\S\), 
    \[
        \Sc^c(\S) \ = \ \sigma_{\scalebox{0.3}{\young(~~~~~~,~~,~)}} +  \sigma_{\scalebox{0.3}{\young(~~~~~~,~,~,~)}} + \sigma_{\scalebox{0.3}{\young(~~~~~,~~~,~)}} + \sigma_{\scalebox{0.3}{\young(~~~~~,~~,~~)}} + \sigma_{\scalebox{0.3}{\young(~~~~~,~~,~,~)}}\; .
    \]
    Proposition \ref{prop:support} tells us that for a partition \(\eta\) to be in \(\supp(\S)\) we need \(\eta \subseteq \lambda = [6,6,2,2]\). When we consider the isomorphic snake matroid \(\S' = \S(1,5,1,2)\) we see that also \(\eta \subseteq \lambda' = [6,5,5,1]\) is necessary. The two snake matroids \(\S\) and \(\S'\) as well as the partitions \(\lambda\) and \(\lambda'\) are depicted in in Figure \ref{fig:support_ex} together with the partitions \(\eta = [3^3]\) and \(\eta' = [3,2^3]\). The partitions \(\eta\) and \(\eta'\) are not in the support of \(\S\) since \(\eta \not \subseteq \lambda\) and \(\eta' \not \subseteq \lambda'\). Note that this is not sufficient to describe the support, as for example the partition \([6,3]\) is contained in both \(\lambda\) and \(\lambda'\), but is not in the support of \(\S\).
\end{example}

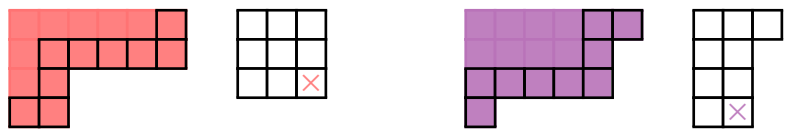
\begin{figure}[ht]
    \centering
    \begin{tikzpicture} 
    \node[Red!110, anchor=south west] at (0,0) {
        \ytableausetup{boxsize = 10pt, boxframe = 1pt}
        \begin{ytableau}
            *(Red) & *(Red) & *(Red) &  *(Red) &  *(Red) &  *(Red)\\
            *(Red) & *(Red) & *(Red) & *(Red) & *(Red) & *(Red)\\
            *(Red) & *(Red)\\
            *(Red) & *(Red)\\
        \end{ytableau}
        \ytableausetup{boxsize = normal, boxframe = normal}
    };
    
    \node[anchor=south west] at (0,0) {
        \ytableausetup{boxsize = 10pt, boxframe = 1pt}
        \begin{ytableau}
            \none & \none & \none  & \none  & \none & *(Red)\\
            \none & *(Red) & *(Red)  & *(Red)  & *(Red) & *(Red)\\
            \none & *(Red)\\
            *(Red) & *(Red)\\
            \end{ytableau}
        \ytableausetup{boxsize = normal, boxframe = normal}
    };

    \node[Violet!110, anchor=south west] at (6,0) {
        \ytableausetup{boxsize = 10pt, boxframe = 1pt}
        \begin{ytableau}
            *(Violet) & *(Violet) & *(Violet) &  *(Violet) &  *(Violet) &  *(Violet)\\
            *(Violet) & *(Violet) & *(Violet) & *(Violet) & *(Violet)\\
            *(Violet) & *(Violet) & *(Violet) & *(Violet)\\
            *(Violet)\\
        \end{ytableau}
        \ytableausetup{boxsize = normal, boxframe = normal}
    };
    
    \node[anchor=south west] at (6,0) {
        \ytableausetup{boxsize = 10pt, boxframe = 1pt}
        \begin{ytableau}
            \none & \none & \none  & \none  & *(Violet) & *(Violet)\\
            \none & \none & \none  & \none  & *(Violet)\\
            *(Violet) & *(Violet) & *(Violet)  & *(Violet)  & *(Violet)\\
            *(Violet)\\
            \end{ytableau}
        \ytableausetup{boxsize = normal, boxframe = normal}
    };

    \node[anchor=south west] at (3,0) {
        \ytableausetup{boxsize = 10pt, boxframe = 1pt}
        \begin{ytableau}
            *(White) & *(White) & *(White)\\
            *(White) & *(White) & *(White) \\
            *(White) & *(White) & *(White) \textcolor{Red}{\times}\\
            \none\\
        \end{ytableau}
        \ytableausetup{boxsize = normal, boxframe = normal}
    };

    \node[anchor=south west] at (9,0) {
        \ytableausetup{boxsize = 10pt, boxframe = 1pt}
        \begin{ytableau}
            *(White) & *(White) & *(White)\\
            *(White) & *(White)\\
            *(White) & *(White)\\
            *(White) & *(White) \textcolor{Violet}{\times}\\
        \end{ytableau}
        \ytableausetup{boxsize = normal, boxframe = normal}
    };
    
\end{tikzpicture}
    \caption{The isomorphic snake matroids \(\S = \S(2,1,5,1)\) and \(\S' = \S(1,5,1,2)\) and partitions \(\lambda = [6^2,2^2]\) in red and \(\lambda' = [6,5^2,1]\) in violet. The partitions \(\eta =  [3^3]\) and \(\eta' = [3,2^3]\) are marked where they extend outside of \(\lambda\) and \(\lambda'\) respectively.}
    \label{fig:support_ex}
\end{figure}

For the special case of the minimal matroid \(\mathsf{T}_{k,n}\) Proposition~\ref{prop:support} together with Example~\ref{ex:recover beta snake} describes all of the support. We again recover \(\Sc(\mathsf{T}_{k,n}) = \sigma_{h^c}\) as in Example~\ref{ex:recover minimal}.

\subsection{Positivity}\label{sec:positivity}
In this section we use Theorem~\ref{thm:combinatorial Sc snakes} to show that the Schubert coefficients~\(d_{\eta^c}(\M)\) are nonnegative whenever \(\M\) is a paving matroid and \(\eta\) is of the form
\[
\eta(m) = [n-k,m+1,1^{k-m-2}]\, .
\]

A central role is played by nested matroids of the form
\[
\Lambda_{k,h,n} \coloneqq \M([n-k,(h-k+1)^{k-1}])
\]
of rank \(k\) on \([n]\). They are also known as panhandle matroids and fall into the class of cuspidal matroids.

Our next results equip us with an explicit description of the Schubert coefficients of those matroids.

\begin{lemma}\label{lem:fullfirst-uniform}
    Let \(\eta(m)\) be the partition defined above. Then
    \begin{align*}
    d_{\eta(m)^c}(\U_{k,n}) \ &= \ \frac{n-k}{k-1}\binom{n-m-2}{n-k}\binom{n-k-1}{m}\, ,\\
    d_{(\eta(m)^t)^c}(\U_{k,n}) \ &= \ \frac{k}{n-k-1}\binom{n-m-2}{k}\binom{k-1}{m}
    \, \text{ and }\\
    d_{\eta(m)^c}(\Lambda_{k,h,n}) \ &= \ \frac{h-k+1}{k-1}\binom{h-m-1}{h-k+1}\binom{h-k}{m} \, 
    .\\
    \end{align*}
\end{lemma}
\begin{proof}
    We begin by proving the second equation from which the first follows directly by duality.
    We decompose the uniform matroid \(\U_{k,n}\) into the sum of all snake matroids contained in it, which are indexed by all possible compositions of \(n-1\) with \(k\) parts. Corollary~\ref{cor:Kostka} implies
    \[
        d_{(\eta(m)^t)^c}(\U_{k,n}) \ = \ \sum_{\mathbf b}K_{\eta(m)^t,\mathbf b}\, .
    \]
    Since \(\eta(m)^t\) has full first column, it is easy to see that the whole sum coincides with \(|\SSYT_{\eta(m)}({=}k)|\). Furthermore, there is a unique way of filling the first column with the numbers from \(1\) to \(k\), this coincides with \(|\SSYT_{[n-k-m-1,1^m]}({\leq} k)|\). The hook-length formula \cite[Corollary~7.21.4]{Stanley-EC2} leads than to the desired statement.
    For the last equation it is enough to show that \(d_{\eta(m)^c}(\Lambda_{k,h,n}) = d_{\theta(m)^c}(\U_{k,h+1})\), where \(\theta(m) = [h-k+1,m+1,1^{k-m-2}]\). This follows by Theorem \ref{thm:combinatorial Sc lattice path}.
\end{proof}

\begin{theorem}\label{thm: nonnegativity paving}
    Let \(\M\) be a connected paving matroid of rank \(k\) on \(n\) elements and \(\eta(m)\) the partition we defined above. Then \(d_{\eta(m)^c}(\M) > 0\) .
\end{theorem}
\begin{proof}
    By the valuativity of \(d_{\eta}\) and \cite[Theorem~5.3]{FerroniSchroeter:2024} we may write
    \begin{equation}
        \label{eq:paving valuativity}
        d_{\eta}(\M) \ = \ d_{\eta}(\U_{k,n}) - \sum_{h} c_h\  d_{\eta}(\Lambda_{k,h,n})
    \end{equation}
    for all partitions \(\eta\), where \(c_h\geq 0\) counts the number of rank \(k-1\) flats in \(\M\) with \(h\geq k\) elements.
    We claim that 
    \begin{equation}\label{eq:key_positivity}
        \frac{d_{\eta(0)^c}(\Lambda_{k,h,n})}{d_{\eta(0)^c}(\U_{k,n})} \ \geq \frac{d_{\eta(m)^c}(\Lambda_{k,h,n})}{d_{\eta(m)^c}(\U_{k,n})}\, .
    \end{equation}
    By weighting with \(c_h\), summing over all values \(h\) and applying the identity \eqref{eq:paving valuativity} we obtain
    \[
       0 \ <\ \frac{\beta(\M)}{\beta(\U_{k,n})} \ = \ \frac{d_{\eta(0)^c}(\M)}{d_{\eta(0)^c}(\U_{k,n})}
        \ \leq \
         \frac{d_{\eta(m)^c}(\M)}{d_{\eta(m)^c}(\U_{k,n})}\, ,
    \]
    where we used that \(d_{\eta(0)^c}(\M) = \beta(\M)\) is positive. It follows from
    Lemma~\ref{lem:fullfirst-uniform} that
    the denominator  \(d_{\eta(m)^c}(\U_{k,n})\) is positive.
    Multiplying by this number reveals that the Schubert coefficient \(d_{\eta(m)^c}(\M)\) is positive. 
     Thus, we are left to prove our claim~\eqref{eq:key_positivity}. By Lemma~\ref{lem:fullfirst-uniform} this is equivalent to the inequality
    \[
    \binom{h-1}{h-k+1}\binom{n-m-2}{n-k}\binom{n-k-1}{m} \geq \binom{h-m-1}{h-k+1}\binom{h-k}{m}\binom{n-2}{n-k}\, ,
    \]
    which simplifies to 
    \[
    \binom{n-m-2}{k-1}\binom{h-1}{k-1} \geq \binom{n-2}{k-1}\binom{h-m-1}{k-1}\, .
    \]
    The last inequality holds true because \((b-i)(a-c-i)\geq (a-i)(b-c-i)\) for all \(0\leq i< m\) and thus
    \(\binom{a-m}{c}\binom{b}{c}\geq \binom{a}{c}\binom{b-m}{c}\) whenever \(a\geq b\).
\end{proof}

\section{Open problems}\label{sec:open problems}
\noindent We want to end this article by listing some open problems and questions which naturally arise from our work.  

\subsection{Schubert coefficients of positroids}
Lattice path matroids are known to have alcoved base polytopes, and hence one may calculate the volume of those polytopes by counting permutations with conditions on their descents, see \cite[Proposition 6.1]{lam-postnikov}. Moreover,  \cite[Theorem 2.1]{lam-postnikov-polypositroids} tells us  that the matroids with alcoved base polytopes are exactly positroids. By Lemma~\ref{lem:important} and RSK we expect that it is feasible to generalize Theorem~\ref{thm:combinatorial Sc lattice path} to positroids.

\begin{problem}
    Given a positroid for example encoded as a Grassmann necklace. Find a description of its Schubert coefficients in terms of standard Young tableaux.
\end{problem}

\subsection{Nonnegativity of Schubert coefficients}
Theorem \ref{thm: nonnegativity paving} provides evidence in support of Conjecture \ref{conj:nonnegativity}. However, this theorem makes assumptions on both the shapes of the partitions and the class of matroids.
It is quite helpful that the Young diagram of \(\eta(m)\) has a full first row to get the expressions in Lemma \ref{lem:fullfirst-uniform}. Even if conjecturally almost all matroids are paving, they are well structured and the
decomposition of the occuring panhandle matroids \(\Lambda_{k,h,n}\) into snakes are in bijection to a decompostion of a uniform matroid into snakes. However, we believe that with additional effort one could use the expression in Corollary \ref{cor:Kostka}, or using the results of \cite{FerroniSchroeter:2024} to extend our results and techniques in the proof of Theorem \ref{thm: nonnegativity paving} to other shapes and more matroids. Two possible first steps in this direction could be the following extensions.

\begin{problem}
    Show nonnegativity of the Schubert coefficients of paving matroids corresponding to any shape with full first row.
\end{problem}

\begin{problem}
    Show nonnegativity of the Schubert coefficients of split matroids corresponding to the shapes \(\eta(m)\) or all shapes with first full row. 
\end{problem}
\subsection{Support of lattice path matroids}
Describing the support of a snake matroid is 
a challenging problem as it is 
equivalent to characterizing the support of a given ribbon Schur function, see \cite{McNamara-Willigenburg}. Furthermore, we know that Lemma~\ref{lem:mcnamara support} and Proposition~\ref{prop:support} are not enough to fully characterize this support. However, it might be the case that the characterization of other classes of matroids is significantly easier. For example, the uniform matroid has full support. We thus formulate the following problem.

\begin{problem}
    Describe the support of nested matroids or other lattice path matroids.
\end{problem}

\subsection{Relating the volume and \texorpdfstring{\(\beta\)}{β}-invariant}
The nonnegativity in Conjecture \ref{conj:nonnegativity} together with Lemma \ref{lem:important} would imply that the inequality \({\beta(\M) \binom{n-2}{k-1} \leq \Volume(\M)}\) holds for any matroid \(\M\) of rank \(k\) on \([n]\). Moreover, by the proof of \cite[Lemma~5.3]{JP}, equality holds for the minimal matroid \(\mathsf{T}_{k,n}\). We remark that this inequality does not directly involve Schubert coefficients and might be of independent interest as it relates two fundamental matroid invariants. We therefore state the following conjecture.

\begin{conjecture}
    Let \(\M\) be a connected matroid of rank \(k\) on \([n]\). Then
    \[
        \beta(\M) \binom{n-2}{k-1} \ \leq \  \Volume(\M) \, ,
    \]
    with equality if and only if \(\M\) is isomorphic to \(\mathsf{T}_{k,n}\) or \(\U_{2,4} \, \).
\end{conjecture}

For connected matroids with nonnegative Schubert coefficients, the strict inequality is equivalent to the support of \(\M\) containing at least one partition in addition to the hook. It follows from Theorem \ref{thm:combinatorial Sc lattice path} that the conjecture holds for lattice path matroids. The conjecture would also imply the classification of sparse paving matroids in terms of Schubert coefficients stated in  \cite[Conjecture 1.3]{JP}.

\bibliographystyle{alpha}
\bibliography{bibliography}

\clearpage
\appendix

\section{Computing the Chow class of a matroid}\label{sec:appendix}
\noindent
In this appendix we demonstrate how to algorithmically compute the Chow class of an arbitrary matroid. SageMath \cite{sagemath} code to compute Chow classes of matroids can be found at
\vspace{5mm}
\begin{center}
    \url{https://github.com/jphamre/Chow-classes-of-matroids-and-SYT} \, .
\end{center}
\vspace{5mm}

Let \(\M\) be the matroid of rank 3 on \(\{1,\ldots,7\}\) whose lattice of cyclic flats \(\mathcal Z(\M)\) is depicted on the left in Figure \ref{fig:cyclic_flats}.
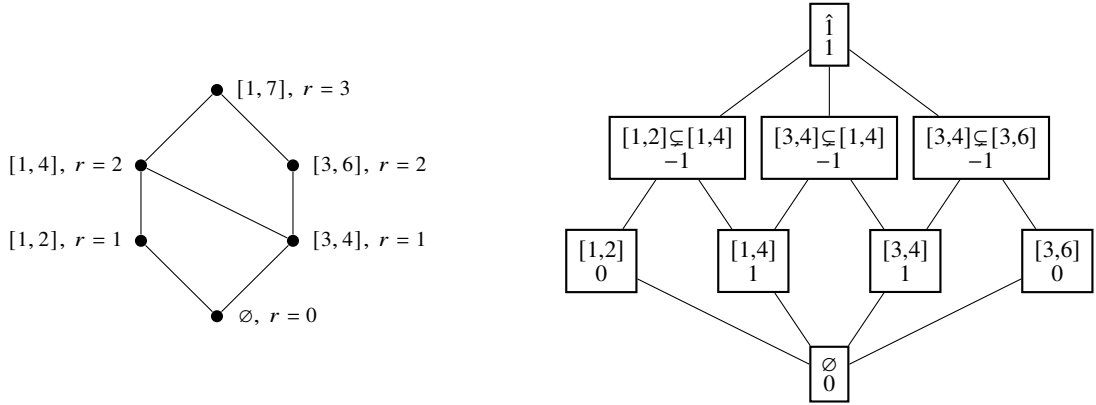
\begin{figure}
    \centering
    \begin{minipage}{0.4\textwidth}
\[\begin{tikzpicture}
\node[circle,scale=0.8, fill=black, inner sep=2pt] (0) at (0,0) {};
\node[circle,scale=0.8, fill=black, inner sep=2pt] (12) at (-1,1) {};
\node[circle,scale=0.8, fill=black, inner sep=2pt] (34) at (1,1) {};
\node[circle,scale=0.8, fill=black, inner sep=2pt] (1234) at (-1,2) {};
\node[circle,scale=0.8, fill=black, inner sep=2pt] (3456) at (1,2) {};
\node[circle,scale=0.8, fill=black, inner sep=2pt] (1234567) at (0,3) {};

\node at (0.8,0) {\tiny{$\varnothing,\ r=0$}};
\node at (-2,1) {\tiny{$[1,2],\ r=1$}};
\node at (2,1) {\tiny{$[3,4],\ r=1$}};
\node at (-2,2) {\tiny{$[1,4],\ r=2$}};
\node at (2,2) {\tiny{$[3,6],\ r=2$}};
\node at (1,3) {\tiny{$[1,7],\ r=3$}};

\draw (0) --(12);
\draw (0) --(34);

\draw (12) --(1234);
\draw (34) --(1234);
\draw (34) --(3456);
\draw (1234) --(1234567);
\draw (3456) --(1234567);
\end{tikzpicture}
\]
\end{minipage}
\hfill
\begin{minipage}{0.5\textwidth}
\[
\begin{tikzpicture}[squarednode/.style={rectangle, draw=black, thick, minimum size=5mm}]
\node[squarednode] (0) at (0,0) {$\substack{\varnothing\\0}$};
\node[squarednode] (12) at (-3,1.5) {$\substack{[1,2]\\0}$};
\node[squarednode] (1234) at (-1,1.5) {$\substack{[1,4]\\1}$};
\node[squarednode] (34) at (1,1.5) {$\substack{[3,4]\\1}$};
\node[squarednode] (3456) at (3,1.5) {$\substack{[3,6]\\0}$};
\node[squarednode] (a) at (-2,3) {$\substack{[1,2]\subsetneq [1,4]\\-1}$};
\node[squarednode] (b) at (0,3) {$\substack{[3,4]\subsetneq [1,4]\\-1}$};
\node[squarednode] (c) at (2,3) {$\substack{[3,4]\subsetneq [3,6]\\-1}$};
\node[squarednode] (1) at (0,4.5) {$\substack{\hat{1}\\1}$};

\draw (0) --(12);
\draw (0) --(1234);
\draw (0) --(34);
\draw (0) --(3456);

\draw (12) --(a);
\draw (1234) --(a);
\draw (1234) --(b);
\draw (34) --(b);
\draw (34) --(c);
\draw (3456) --(c);

\draw (a) --(1);
\draw (b) --(1);
\draw (c) --(1);

\end{tikzpicture}
\]    
\end{minipage}
    \caption{On the left, the lattice of cyclic flats \(\mathcal{Z}(\M)\) of \(\M\). We specify also the rank of each cyclic flat. On the right, its cyclic chain lattice \(\mathcal C_{\mathcal Z}(\M)\). For convenience, we omit to write the top and bottom element of every chain, which are always \(\varnothing\) and \([1,7]\), respectively. Below a given element \(\N\), we also write \(\mu_{\mathcal C_{\mathcal Z}(\M)}(\N,\hat{1})\)}
    \label{fig:cyclic_flats}
\end{figure}
We use Theorem \ref{thm:bases_valuation} and Theorem \ref{thm:FinkSpeyer}  to write
\[
    \Sc(\M) = \sum_{\N \in \mathcal N_{n,k}} c(\N,\M) \Sc(\N)\, .
\]
To find the value of the coefficients \(c(\N,\M)\), we follow Hampe's algorithmic method \cite{Hampe}. We consider the cyclic chain lattice \(\mathcal C_{\mathcal Z}(\M)\), i.e., the poset of all chains in \(\mathcal Z(\M)\) starting at the minimal element \(\hat{0}_{\mathcal Z(\M)} = \varnothing\) and ending at the maximal element \(\hat{1}_{\mathcal Z(\M)} = [1,7]\), ordered by inclusion and with an artificial top element \(\hat{1}\). See the poset on the right in Figure \ref{fig:cyclic_flats}. Each of the elements in \(\mathcal C_{\mathcal Z}(\M)\) except for \(\hat{1}\) corresponds to a nested matroid \(\N\), as a matroid is nested if and only if its lattice of cyclic flats is a chain. Moreover, \(c(\N,\M) = -\mu_{\mathcal C_{\mathcal Z}(\M)}(\N,\hat{1})\), where \(\mu_{\mathcal C_{\mathcal Z}(\M)}\) is the M\"{o}bius function of \(\mathcal C_{\mathcal Z}(\M)\).
This means that the Chow class of \(\M\) is equal to 
\begin{equation}
    \label{eq:nested decomp}
    \Sc(\M) = 3\Sc\left( \scalebox{0.3}{\young(::~~,:~~~,~~~~)}\right) - \Sc\left( \scalebox{0.3}{\young(::~~,~~~~,~~~~)}\right) - \Sc\left( \scalebox{0.3}{\young(:~~~,:~~~,~~~~)}\right)\, ,
\end{equation}
where each nested matroid \(\N\) is isomorphic to one of the lattice path matroids appearing on the right hand side. Using Corollary \ref{cor:combinatorial Sc nested} we compute the Chow classes of the nested matroid above,
\begin{align*}
    \Sc\left( \scalebox{0.3}{\young(::~~,:~~~,~~~~)}\right) \ &= \  3s_{\scalebox{0.3}{\young(~~~~,~~)}} +  5s_{\scalebox{0.3}{\young(~~~,~~~)}} + 2s_{\scalebox{0.3}{\young(~~~~,~,~)}} + 5s_{\scalebox{0.3}{\young(~~~,~~,~)}} + s_{\scalebox{0.3}{\young(~~,~~,~~)}}\\
    \Sc\left( \scalebox{0.3}{\young(::~~,~~~~,~~~~)}\right) \ &= \  5s_{\scalebox{0.3}{\young(~~~~,~~)}} +  7s_{\scalebox{0.3}{\young(~~~,~~~)}} + 3s_{\scalebox{0.3}{\young(~~~~,~,~)}} + 6s_{\scalebox{0.3}{\young(~~~,~~,~)}} + s_{\scalebox{0.3}{\young(~~,~~,~~)}}\\
    \Sc\left( \scalebox{0.3}{\young(:~~~,:~~~,~~~~)}\right) \ &= \  3s_{\scalebox{0.3}{\young(~~~~,~~)}} +  6s_{\scalebox{0.3}{\young(~~~,~~~)}} + 2s_{\scalebox{0.3}{\young(~~~~,~,~)}} + 6s_{\scalebox{0.3}{\young(~~~,~~,~)}} + s_{\scalebox{0.3}{\young(~~,~~,~~)}} \, .\\
\end{align*}
Now \eqref{eq:nested decomp} gives
\[
    \Sc(\M) \ = \  s_{\scalebox{0.3}{\young(~~~~,~~)}} +  2s_{\scalebox{0.3}{\young(~~~,~~~)}} + s_{\scalebox{0.3}{\young(~~~~,~,~)}} + 3s_{\scalebox{0.3}{\young(~~~,~~,~)}} + s_{\scalebox{0.3}{\young(~~,~~,~~)}} \, .
\]
Alternatively, one can further decompose each of the nested matroids in \eqref{eq:nested decomp} into the snake matroids contained in their skew shape. After grouping similar terms we get
\begin{align*}
    \Sc(\M) \ &= \ \Sc\left( \scalebox{0.3}{\young(::~~,:~~,~~)}\right) + 
    \Sc\left( \scalebox{0.3}{\young(:::~,:~~~,~~)}\right) + 
    \Sc\left( \scalebox{0.3}{\young(::~~,::~,~~~)}\right) + 
    \Sc\left( \scalebox{0.3}{\young(:::~,::~~,~~~)}\right)\\
    & \quad + 
    \Sc\left( \scalebox{0.3}{\young(:::~,:::~,~~~~)}\right)
    - \Sc\left( \scalebox{0.3}{\young(::~~,~~~,~)}\right) -
    \Sc\left( \scalebox{0.3}{\young(:::~,~~~~,~)}\right) -
    \Sc\left( \scalebox{0.3}{\young(:~~~,:~,~~)}\right)\, .
\end{align*}
By Theorem \ref{thm:ribbon snake} we get the following expression for the Poincaré dual of \(\Sc(\M)\), 
\[
    \Sc^c(\M) \ = \ 
    s_{\scalebox{0.3}{\young(::~~,:~~,~~)}} + 
    s_{\scalebox{0.3}{\young(:::~,:~~~,~~)}} + 
    s_{\scalebox{0.3}{\young(::~~,::~,~~~)}} + 
    s_{\scalebox{0.3}{\young(:::~,::~~,~~~)}} + 
    s_{\scalebox{0.3}{\young(:::~,:::~,~~~~)}} - 
    s_{\scalebox{0.3}{\young(::~~,~~~,~)}} -
    s_{\scalebox{0.3}{\young(:::~,~~~~,~)}} - 
    s_{\scalebox{0.3}{\young(:~~~,:~,~~)}} \, .
\]
Using this algorithm, we were able to compute the Chow class of much larger matroids than previously achieved. More specifically, we were able to obtain explicit expressions for all matroids in the matroid catalog, with ground sets up to \(8\) elements, and all paving matroids up to \(15\) elements.

\begin{proposition}
    The Schubert coefficient \(d_\eta(\M)\) is nonnegative for any \(\eta\) whenever
    \begin{itemize}
        \item \(\M\) is a matroid on a ground set with at most \(8\) elements, or
        \item \(\M\) is a paving matroid on a ground set with at most \(15\) elements. 
    \end{itemize}
\end{proposition}

\end{document}